\documentclass[12pt,centertags,oneside, reqno]{amsart}
\usepackage{amsmath,amstext,amsthm,amscd,typearea}

\usepackage{amssymb}
\usepackage{braket}
\usepackage[all]{xy}
\usepackage{cancel}
\usepackage{microtype}

\usepackage{amscd,amsxtra,calc}
\usepackage{cmmib57}
\usepackage{url}
\usepackage{ulem}
\usepackage{xcolor}
\usepackage[pagebackref]{hyperref}

\usepackage[a4paper,width=16.2cm,top=3cm,bottom=3cm]{geometry}

\numberwithin{equation}{section}

\theoremstyle{plain}
\newtheorem{thm}{Theorem}[section]
\newtheorem{theorem}[thm]{Theorem}

\newtheorem{lemma}[thm]{Lemma}
\newtheorem{corollary}[thm]{Corollary}
\newtheorem{proposition}[thm]{Proposition}
\theoremstyle{definition}
\newtheorem{remark}[thm]{Remark}

\newtheorem{definition}[thm]{Definition}
\newtheorem{claim}[thm]{Claim}
\newtheorem{assumption}[thm]{Assumption}

\newtheorem{example}[thm]{Example}

\newtheorem{question}[thm]{Question}

\numberwithin{equation}{section}

\newcommand{\Pic}{{\rm Pic}}

\newcommand{\Spec}{{\rm Spec \,}}

\newcommand{\Gal}{{\rm Gal}}

\newcommand{\C}{{\mathbb C}}

\newcommand{\G}{{\mathbb G}}

\renewcommand{\P}{{\mathbb P}}
\newcommand{\Q}{{\mathbb Q}}
\newcommand{\R}{{\mathbb R}}

\newcommand{\Z}{{\mathbb Z}}

\newcommand{\id}{{\rm id\hspace{.1ex}}}

\newcommand{\Aut}{{\rm Aut\hspace{.1ex}}}

\newcommand{\ssec}{\subsection}

\newcommand{\ol}{\overline}
\newcommand{\ti}[1]{\tilde{#1}}

\newcommand{\vast}{\bBigg@{4}}
\newcommand{\Vast}{\bBigg@{5}}
\newcommand{\wt}{\widetilde}

\newcommand{\bA}{\mathbf{A}}

\newcommand{\bk}{\mathbf{k}}

\newcommand{\cS}{\mathcal{S}}

\newcommand{\gO}{\Omega}
\newcommand{\gS}{\Sigma}

\newcommand{\go}{\omega}
\newcommand{\gs}{\sigma}

\newcommand{\Amp}{\mathrm{Amp}}

\newcommand{\KAut}{\mathrm{KAut}}

\newcommand{\Bir}{\mathrm{Bir}}

\newcommand{\Dec}{\mathrm{Dec}}

\newcommand{\Fix}{\mathrm{Fix}}
\newcommand{\GL}{\mathrm{GL}}

\newcommand{\group}{\mathrm{group}}

\newcommand{\Ima}{\mathrm{Im}}
\newcommand{\Ine}{\mathrm{Ine}}

\newcommand{\Ker}{\mathrm{Ker}}

\newcommand{\MW}{\mathrm{MW}}
\newcommand{\NS}{\mathrm{NS}}
\newcommand{\Nef}{\mathrm{Nef}}

\newcommand{\sm}{\mathrm{sm}}

\newcommand{\torsion}{\mathrm{torsion}}

\newcommand{\cnec}{\mathrel{:=}}
\newcommand{\cto}{\circlearrowleft}

\newcommand*\eto{%
	\xrightarrow[]{\raisebox{-0.25 em}{\smash{\ensuremath{\sim}}}}%
}
\newcommand{\hto}{\hookrightarrow}

\title [Infinitely many real forms]{Smooth projective surfaces 
with infinitely many \\ real forms}

\author{{ Tien-Cuong Dinh}}
\address{Department of Mathematics, National University of Singapore, 10, Lower Kent Ridge Road, Singapore 119076}
\email{matdtc@nus.edu.sg}

\author{{ C\'ecile Gachet}}
\address{Université Côte d’Azur, CNRS, LJAD, France}
\email{gachet@unice.fr}

	\author{Hsueh-Yung Lin}
	\address{Department of Mathematics, National Taiwan University, 
	and National Center for Theoretical Sciences,
	Taipei, Taiwan.}
	\email{hsuehyunglin@ntu.edu.tw}

\author{\\ {Keiji Oguiso}}
\address{Graduate School of Mathematical Sciences, The University of Tokyo, Meguro Komaba 3-8-1, Tokyo, Japan, and National Center for Theoretical Sciences, Mathematics Division, National Taiwan University, Taipei, Taiwan}
\email{oguiso@ms.u-tokyo.ac.jp}

\author{{ Long Wang}}
\address{Graduate School of Mathematical Sciences, The University of Tokyo, Meguro Komaba 3-8-1, Tokyo, Japan}
\email{wangl11@ms.u-tokyo.ac.jp}

\author{{ Xun Yu}}
\address{Center for Applied Mathematics, Tianjin University, 92 Weijin Road, Nankai District, Tianjin 300072, P. R. China.}
\email{xunyu@tju.edu.cn}

\thanks{T.-C. Dinh is supported by the NUS grant A-0004285-00-00 
and MOE grant MOE-T2EP20120-0010. 
H.-Y. Lin is supported by the Ministry of Education Yushan Young Scholar Fellowship (NTU-110VV006) and the National Science and Technology Council (110-2628-M-002-006-). 
K. Oguiso is supported by JSPS Grant-in-Aid 20H00111, 20H01809, and by NCTS Scholar Program. 
L. Wang is supported by JSPS KAKENHI Grant Number 21J10242. 
X. Yu is supported by NSFC (No.12071337, No. 11701413, and No. 11831013).
}


\begin{document}
	
	\maketitle
	\begin{abstract}  
	The aim of this paper is twofold.
	First of all, we confirm a few basic criteria of the finiteness of real forms of a given smooth complex projective variety, in terms of the Galois cohomology set of the discrete part of the automorphism group, the cone conjecture and the topological entropy. We then apply them to show that a smooth complex projective surface has at most finitely many non-isomorphic real forms unless it is either rational or a non-minimal surface birational to either a K3 surface or an Enriques surface. 
	In the second part of the paper, 
	we construct an Enriques surface whose blow-up at one point admits infinitely many non-isomorphic real forms. 
	This answers a question of Kondo to us and also shows the three exceptional cases really occur.
	\end{abstract}
	
	\section{Introduction}

	Let $V$ be a complex algebraic variety.
	A real from of $V$ is a real algebraic variety $W$
	such that
	$$V \simeq W \times_{\Spec \R} \Spec \C.$$
	In his seminal work~\cite{Le18},
	Lesieutre constructed the first  
	smooth complex projective varieties
	with infinitely many non-isomorphic real forms.
	Later, Dinh--Oguiso 
	constructed the first smooth projective surfaces 
	with the same property~\cite{DO19}.
	More examples were constructed in~\cite{DOY22, DOY23}.
	
	For most examples,
	it is also proven in \textit{loc. cit.} that the discrete part of the automorphism group 
	$\Aut(V)/\Aut^0(V)$ is not finitely generated.
	This motivates the following question.
	
	\begin{question}\label{que-RfvsAut}
	Let $V$ be a complex projective variety.
	Suppose that $V$ has infinitely many
	real forms. 
	Does $V$ have large automorphism group $\Aut(V)$
	or group action $\Aut(V) \cto V$ with high complexity?
	\end{question}
	
	Depending on what we mean by
	"large" automorphism group and "high complexity",
	there may be many ways to approach and interpret
	Question~\ref{que-RfvsAut}.
	We will see two
	answers to Question~\ref{que-RfvsAut},
	in \S\ref{ssec-RfvsAut} and \S\ref{ssec-RfvsNef}
	respectively.
	
		\ssec{Real forms and automorphism groups}\label{ssec-RfvsAut}

	Our first answer to Question~\ref{que-RfvsAut}
	relies on the following theorem
	asserted in~\cite[Appendix D]{DIK00}.
	We will provide a proof in Section~\ref{sct2}
	(see Remark~\ref{rem220}).

	\begin{theorem}\label{thm12}
		Let $V$ be a complex projective variety with a real form. If 
		$$H^1({\rm Gal}(\C/\R), {\rm Aut}(V)/{\rm Aut}^0(V))$$ 
		is finite, then $V$ has only finitely many non-isomorphic real forms. Moreover, 
		the set of non-isomorphic real forms of a complex projective variety is at most countable.
	\end{theorem}

    \begin{remark}
    Without the 
    projectivity assumption,
    there exist counterexamples to the last statement of
    Theorem~\ref{thm12}, 
    already among affine surfaces by Bot's construction~\cite{Bo21}.
    \end{remark}

	Let $\NS(V)$ denote the N\'eron-Severi group of 
	a projective variety $V$; 
	it is a finitely generated abelian group.
	We will apply Theorem~\ref{thm12} to prove 
	the following corollary. 
	
	\begin{corollary}\label{cor-ent0}

		Let $V$ be a complex projective variety.
		If $\Aut(V)/\Aut^0(V)$, or more generally
		the image of the pullback action 
		$$\rho : \Aut(V)/\Aut^0(V) \to \GL(\NS(V)/\torsion)$$
		is virtually solvable,
		then $V$ has at most finitely many 
		non-isomorphic real forms.
	\end{corollary}
	
	Corollary~\ref{cor-ent0} thus 
	provides an answer to Question~\ref{que-RfvsAut}.
	Thanks to
	Tits' alternative, 
	we obtain the following
	more explicit consequences.

	\begin{corollary}\label{cor-ent0bis} 
	    Let $V$ be a complex projective variety
	    with infinitely many 
		non-isomorphic real forms.
		The following statements hold:
		\begin{enumerate}
		    \item $\Aut(V)/\Aut^0(V)$
		    contains a non-abelian free group.
		    \item Assume that $V$ is smooth. Then
		    $V$ admits an automorphism of positive entropy.
		\end{enumerate}
	\end{corollary}
	
	Corollary~\ref{cor-ent0bis}.(2)
	generalizes~\cite[Theorem 1]{Be16}
	from rational surfaces 
	to arbitrary smooth projective varieties and a result of \cite{Ki20}, which are based on Theorem~\ref{thm12}.
	All these corollaries show that if a 
	complex projective variety
	admits infinitely many real forms,
	then its automorphism group
	is necessarily quite complicated.

	\ssec{Real forms and the action on the nef cone}\label{ssec-RfvsNef}

	Inside the $\R$-vector space $\NS(V)\otimes_{\Z} \R$, 
	let $\Amp(V)$ and $\Nef(V)$ 
	denote the ample cone and the nef cone of $V$ respectively.
	Let ${\rm Nef}^{+} (V)$ be the rational hull of $\Nef(V)$, that is, the convex hull of the set 
	$$\left(\NS(V)\otimes_{\Z} \Q\right) \cap \Nef(V).$$ 
	We also let 
	$${\rm Aut}^{*}(V) := {\rm Im}({\rm Aut}(V) \to {\rm GL}(\NS(V)/\torsion))$$
	be the image under the natural action. Then ${\rm Aut}^{*}(V)$ preserves ${\rm Nef}^{+} (V)$. 
	In terms of the action ${\rm Aut}^{*}(V) \cto {\rm Nef}^{+} (V)$,
	the following is another answer to Question~\ref{que-RfvsAut}.

	\begin{theorem}\label{thm13}
		Let $V$ be a complex projective variety such that
		${\rm Nef}^{+} (V)$ 
		contains a rational 
		polyhedral  cone $\Sigma$
		satisfying
		$${\rm Aut}^{*}(V) \cdot \Sigma \supset \Amp(V).$$ 
		For instance, this is the case
		when $\Nef^+(V)$ is a rational polyhedral cone,
		or more generally when
		$V$ satisfies the cone conjecture, in the sense that the natural action of ${\rm Aut}^{*}(V)$
		on $\Nef^+(V)$ has a rational polyhedral fundamental domain.

		Then $V$ has at most finitely many mutually non-isomorphic real forms. 
		In particular, this is the case where $V$ is a minimal surface of Kodaira dimension zero by Sterk \cite{St85}, Namikawa \cite{Na85} 
		and Kawamata \cite[Theorem 2.1]{Ka97}.
	\end{theorem}
	
    Essentially the same result as Theorem~\ref{thm13} was asserted by~\cite{CF19}. 
    We will provide a proof of Theorem~\ref{thm13} in Section~\ref{sct2}
    (see Remark~\ref{rem22}).

	\ssec{Smooth projective surfaces with infinitely many real forms}

	Both Corollary~\ref{cor-ent0} and Theorem~\ref{thm13} 
	will be applied to
	complete the proof
	of the following folklore result for surfaces.

	\begin{theorem}\label{thmRFsurf} 
			 Let $S$ be a smooth complex projective surface. Assume that $S$ has infinitely many mutually non-isomorphic real forms. Then $S$ is either rational or a non-minimal surface birational to either a K3 surface or an Enriques surface.
	\end{theorem}
	
	Theorem~\ref{thmRFsurf} should be known to experts. 
	We will give a proof in Section~\ref{sct3}
	(due to Remarks~\ref{rem220} and~\ref{rem22}), 
	along the line explained by \cite{DIK00} and \cite{CF19} with clarifications for the sake of completeness. 
	Along the way, we also prove some results 
	which hold in arbitrary dimension 
	(e.g. Proposition~\ref{prop32}). 
	
	Our previous result of~\cite{DOY23} shows that there is a smooth projective rational surface $S$ with infinitely many mutually non-isomorphic real forms, which answers a question by \cite{Kh02}. There is also a smooth projective surface $S$ which is a blow-up of some K3 surface at one point such that $S$ admits infinitely many mutually non-isomorphic real forms. Such a surface $S$ is constructed first by \cite{DOY23} after \cite{DO19}, answering a question of Mukai to us.

	Given the above surface examples and Theorem~\ref{thmRFsurf},
	Kondo asked us whether
	there exists a surface $S$ as in Theorem~\ref{thmRFsurf}
	which is birational to an Enriques surface.
    The second half of our paper (starting from Section~\ref{sec-Kummer})
    is entirely devoted to
    the construction of such examples.

	\begin{thm}\label{thm11}
	There is a blow-up $Z$ of an Enriques surface at one point 
			such that
			\begin{enumerate}
			    \item $Z$ admits infinitely many non-isomorphic real forms. 
			    \item $\Aut(Z)$ is not finitely generated.
			\end{enumerate}
	\end{thm}
	
	\begin{remark}
	For (2), surfaces whose automorphism groups are not finitely generated have been 
	previously constructed among blow-ups 
	of Enriques surfaces at at least two points~\cite{KO19, Wa21}.
	\end{remark}
	
	Our construction 
	is inspired by \cite{Le18}, \cite{DO19}, \cite{DOY23} and \cite{Mu10}. 
	We prove Theorem~\ref{thm11} in Sections~\ref{sct5},
	and refer to Theorem~\ref{thm32} and Remark~\ref{rem_dense}
	for more precise statements.
	By Theorem~\ref{thm11}, together with our previous work~\cite{DOY23}, 
	we conclude that
	the three cases in Theorem~\ref{thmRFsurf} all occur.

	\medskip\noindent
	{\bf Acknowledgements.} We would like to thank Professor J.-H. Keum for substantial discussions in our earlier works and Professors B. Lian and S. Kondo for valuable discussions, which are much reflected in this paper. 
	C.G. would like to thank JSPS Summer Program for providing the opportunity to visit K.O. and L.W. in Tokyo, where part of this paper was written. 
	L.W. thanks Department of Mathematics at National University of Singapore, Professor D.-Q. Zhang 
	and Doctor J. Jia 
	for warm hospitality.
	
	\medskip\noindent
	{\bf Notation and convention.}  
	We work over the field $\C$ of complex numbers, and refer to \cite{BHPV04} for basic definitions and properties of complex projective surfaces. 
	
	In this paper,
	by a point of a projective variety $V$ over ${\mathbb C}$, we always mean a point of $V({\mathbb C})$, i.e., a ${\mathbb C}$-valued point of $V$, except a generic point by which we always mean a generic point in the scheme theoretic sense. A \textit{locally algebraic group} is a group scheme locally of finite type over a field.

	For every scheme $V$ over a field $\Bbbk$ 
	(in our paper $\Bbbk$ will be $\R$ or $\C$),
	we let $\Aut(V/\Bbbk)$ denote the 
	group of biregular automorphisms of $V$ over $\Bbbk$.
	We also write $\Aut(V) = \Aut(V/\Bbbk)$
	if there is no risk of confusion and, unless stated otherwise, we regard $\Aut(V) = \Aut(V/\Bbbk)$ as an abstract group (not as a group scheme). Note that if $V$ is defined over ${\mathbb R}$ and $\Aut(V/{\mathbb C}) = \{\id_V\}$, then the Galois group ${\rm Gal}({\mathbb C}/{\mathbb R})$ acts trivially on the abstract group $\Aut(V/{\mathbb C})$, whereas it acts as an involution on the group scheme $\Aut(V/{\mathbb C}) \to {\rm Spec}\, {\mathbb C}$.
	Given a morphism $f : X \to B$ of varieties,
	we define $\Aut(X/B)$ (resp. $\Bir(X/B)$)
    as the group of automorphisms (resp. birational automorphisms) $\phi$
    preserving $f$ and acting as the identity on $B$.

	For a complex variety $V$,
	we define the decomposition group and the inertia group of subsets $W_1,\ldots,W_n \subset V$ by
	\begin{align*}
	{\rm Dec} (V, W_1,\ldots,W_n) 
	&:= \{ f \in {\rm Aut}(V) \,|\,\forall i,\,f(W_i) = W_i \}, \quad\,\\
	{\rm Ine} (V, W_1,\ldots,W_n) 
	&:= \{ f \in {\rm Dec} (V, W_1,\ldots,W_n) \,|\,\forall i,\, f_{W_i} = \id_{W_i}\}.
	\end{align*}
	Note then that
	$${\rm Dec} (V, W_1, \ldots, W_n) \subset {\rm Dec} (V, \cup_{i=1}^{n} W_i),$$
	and for an irreducible decomposition $W = \cup_{i=1}^{n} W_i$ of an algebraic set $W \subset V$, 
	$$[{\rm Dec} (V, \cup_{i=1}^{n} W_i) : {\rm Dec} (V, W_1, \ldots, W_n)] \le |S_n| = n!.$$
	For an automorphism $f \in {\rm Aut}(V)$, we denote the set of fixed point of $f$ by
	$$V^f := \{x \in V({\mathbb C})\,|\, f(x) = x\}.$$
	
	We refer to e.g.~\cite[Section I.5]{Se02} 
	for the basic facts on the group cohomology set $H^1(G, B)$ of a $G$-group $B$. 
	In this paper, we only need the non-trivial simplest case where 
	$$G = G_{\C/\R} \cnec {\rm Gal}({\mathbb C}/{\mathbb R}) \simeq {\mathbb Z}/2{\mathbb Z}.$$
	
	
	\section{Two basic criteria of finiteness of real forms}\label{sct2}
	
	In this section, 
	we first recall the notion of real forms and some classical results
	due to Borel, Serre, and Weil,
	in order to fix some notations.
	We will then prove Theorems~\ref{thm12} and~\ref{thm13}. 
	
	\ssec{Real forms and real structures} 
	\hfill
	
	Throughout the paper,
	$c : \C \to \C$
	denotes the complex conjugate, so
	$$G_{\C/\R} = \Gal(\C/\R) = \left\{ \id_{\C}, c \right\}.$$
	
	Let $V$ be a scheme over $\C$
	and let $\pi : V \to \Spec \C$ be the structural morphism.
	
	\begin{definition}\label{def21} 
		\hfill
		\begin{enumerate}
			\item A {\it real form} of $V$
			is a scheme $W$ over $\R$ such that
			$$V \simeq W \times_{\Spec \R} \Spec \C$$
			over $\Spec \C$.
			\item
			A {\it real structure} of $V$ is an anti-holomorphic involution  
			$$\imath : V \to V,$$ 
			namely $\imath$ is an automorphism over $\Spec \R$ such that
			$$\imath^2 = \id_V \ \ \text{ and } \ \ \pi \circ \imath = c \circ \pi.$$
		\end{enumerate}
		Two real forms $W$ and $W'$ are equivalent
	if they are isomorphic over $\Spec \R$.
	Two real structures $\imath$ and $\imath'$ on $V$ are said to be equivalent if 
	$\imath' = h \circ \imath \circ h^{-1}$ 
	for some $h \in {\rm Aut}(V/\C)$. 
	\end{definition} 
	
	The real structure associated to a real form $W$ of scheme $V$ over $\C$ 
	is defined as
	$$\imath_{W} \cnec  \id_W \times c : V \to V,$$ 
	if one fixes an identification $V = W \times_{\Spec \R} \Spec \C$.
	Assume that $V$ is a quasi-projective variety. 
	As a consequence of Galois descent,
	the map $W \mapsto \imath_W$ defines a one-to-one correspondence
	\begin{equation}\label{1-1R}
		\xymatrix{
			\left\{ \text{Real forms on }V \right\}/\simeq \ \ \ar@{<->}[rr] & & 
			\ \ \left\{ \text{Real structures on }V \right\}/\simeq. \\
		}
	\end{equation}
	
	\begin{example}\label{ex-RF}
		\hfill

		\begin{enumerate}
			\item Let $W$ be a real form of a complex scheme $V$.
			Then $G_{\C/\R}$ acts naturally on the group scheme $\Aut(V/\C)$ by
			\begin{equation}\label{eqn-RstAut}
				c \cdot f = \imath_W \circ f \circ \imath_W,
			\end{equation}
			which we fix throughout the paper.
			If $V$ is a projective complex variety,
			then $\Aut(V/\C)$ is a locally algebraic group over $\C$ and $\Aut(W/\R)$ is a real form of it~\cite[Theorem 3.7]{MO67}. See also \cite[Section 5.6]{FGIKNV}.
			The associated real structure on $\Aut(V)$ 
			is defined by~\eqref{eqn-RstAut}.
			
			\item Let $V_\R$ be a real scheme and let $V$ be its complexification.
			Let $\imath : V \to V$ be the associated real structure.
			For every $f \in \Aut(V/\C)$ such that 
			\begin{equation}\label{eqn-iconj}
				c \cdot f \cnec  \imath \circ f \circ \imath = f^{-1},
			\end{equation}
			the composition 
			$$\imath \circ f : V \to V$$ 
			defines a real structure on $V$.
			Condition~\eqref{eqn-iconj} is equivalent to the property that
			$$\phi : G_{\C/\R} \to \Aut(V)$$
			defined by $\phi(\id_{\C}) = \id_{\C}$ and $\phi(c) = f$ is a $1$-cocycle 
			where the $G_{\C/\R}$-action on $\Aut(V)$ is defined by~\eqref{eqn-RstAut}.
			We call $\imath \circ f$ the real structure twisted by $\phi$,
			and let $V_\phi$ denote
			the complex scheme $V$ endowed with the new
			$G_{\C/\R}$-action defined by $c \cdot v \cnec \imath(f(v))$ for all $v \in V$.
			We also let $V_{\R,\phi}$ denote the corresponding real form.
			
			\item 
			We continue the above example,
			and assume moreover that $V_{\R}$ is a real group scheme:
			then $V$ is a complex group scheme.
			We verify that 
			the group laws of $V_\phi$, viewed as morphisms over $\C$,
			are $G_{\C/\R}$-equivariant,
			so they descend to group laws on the real form $V_{\R,\phi}$,
			giving it a group scheme structure over $\R$.
			Finally, note that
			if $V_{\R}$ (or equivalently $V$) is an algebraic group,
			then so is $V_{\R,\phi}$.
			Moreover, 
			since for algebraic groups, 
			the property of being linear (resp. connected) 
			does not depend on the base field,
			if $V_{\R}$ is linear (resp. connected)
			then so is $V_{\R,\phi}$.

		\end{enumerate}
	\end{example}

	We can also describe the set of real forms up to equivalence
	using Galois cohomology~\cite[Page 124, Proposition 5]{Se02}.

	\begin{thm}\label{thm21}
		Let $V$ be a complex quasi-projective variety having a real form $W$ 
		with real structure $\imath_W$. 
		Then there are natural bijective correspondences between
		the following three sets:
		\begin{enumerate}
			\item The set of real forms of $V$ up to isomorphism as varieties over ${\mathbb R}$;
			\item The set of real structures 
			on $V$ up to equivalence;
			\item The Galois cohomology set 
			$$H^1(G_{\C/\R}, {\rm Aut}(V)),$$ 
			where the action of $G_{\C/\R}$ on ${\rm Aut}(V)$ is given by $f \mapsto \imath_W \circ f \circ \imath_W$. 
		\end{enumerate}
	\end{thm}

	For later use,  we say that a subvariety $W$ on $V$ (resp. a morphism $f : V \to U$) is defined over ${\mathbb R}$ with respect to the real form $V_{\R}$ (resp. real forms $V_{\R}$ and $U_{\R}$) if there is an object $W_{\mathbb R}$ on $V_{\mathbb R}$ (resp. a morphism $f_{{\mathbb R}} : V_{{\mathbb R}} \to U_{{\mathbb R}}$) such that $W = W_{{\mathbb R}} \times_{{\rm Spec}\, {\mathbb R}}{\rm Spec}\, {\mathbb C}$ (resp. $f = f_{{\mathbb R}} \times \id_{{\rm Spec}\, {\mathbb C}}$ for some morphism $f_{\mathbb R} : V_{{\mathbb R}} \to U_{{\mathbb R}}$). 
	We say that a subvariety $W$ on $V$ is defined over ${\mathbb R}$ with respect to a real structure of $V$, if $W$ defined over ${\mathbb R}$ with respect to the corresponding real form. Similarly, we have the definition of a morphism $f : V \to U$ defined over ${\mathbb R}$ with respect to two real structures of $V$ and $U$. When a real structure $\imath$ of $V$ is fixed, by abuse of terminology, a complex point $x$ of $V$ is called a real point if $x \in V^{\imath}$, i.e., if the support of $x$ is fixed under $\imath$. Note that $V({\mathbb C})^{\imath} = V_{{\mathbb R}}({\mathbb R})$ as sets.

	\ssec{Some finiteness results of Galois cohomology}
	\hfill

	Recall that a group $H$ is said to be {\it polycyclic} if it is solvable and every subgroup of $H$ is finitely generated. 
	
	The following proposition is well-known. In our applications, the $G$-group $H$ in Proposition \ref{prop21} will be mostly a subgroup or a quotient group of ${\rm Aut}(V)$ of a complex projective variety $V$ having a real form $V_{0}$ with real structure $c_{0}$, to which the action of $c_0$ by conjugation restricts or extends.

	\begin{proposition}\label{prop21}
		Set $G := {\rm Gal}({\mathbb C}/{\mathbb R})$. 
		Let $H$ be a $G$-group. 
		\begin{enumerate}
			\item Suppose that the $G$-group $H$ is arithmetic,
			in the sense that there exists a linear $G$-group $L_{\Q}$ over $\Q$
			such that $H$ embeds $G$-equivariantly into $L_{\Q}$
			as an arithmetic subgroup.
			Then $H^1(G, H)$ is finite.
			
			\item If $H$ has a filtration consisting of normal $G$-subgroups 
			$N_i$ of $H$
			$$\{1_H\} = N_s \le N_{s-1} \le \ldots \le N_{1} \le N_0 = H$$
			such that $H^1(G,N_i/N_{i+1})$ is finite for any $G$-action on $N_i/N_{i+1}$ (this is the case when e.g.
			$N_i/N_{i+1}$ is either a finitely generated abelian group or a finite group),
			then $H^1(G, H)$ is a finite set. 
			
			\item 
			Let $H$ be a $G$-group which is virtually polycyclic,
			namely, $H$ admits a finite index polycyclic subgroup $N \le H$
			(without assuming that the $G$-action preserves $N$), 
			then $H^1(G, H)$ is a finite set.
			
			\item Assume that the $G$-action on $H$ is trivial. Then the cardinality of $H^1(G, H)$ coincides with the cardinality of the set of conjugacy classes of involutions with $1_H$ in $H$. 
		\end{enumerate}
	\end{proposition}
	
	\begin{proof} (1) is proved by \cite[Th\'eor\`eme 6.1]{BS64}. 
		(2) is stated by \cite[D.1.7, Appendix D]{DIK00} and rigorously restated and proved by \cite[Lemma 4.9]{CF19}. 
		
		Now we prove (3). 
		Suppose $N$ is a polycyclic subgroup of $H$ of finite index. 
		Up to replacing $N$ by 
		$$\bigcap_{h\in H} h^{-1}Nh,$$ 
		which is still a finite index subgroup of $H$, 
		we can assume that $N$ is normal in $H$.
		Up to replacing $H$ by 
		$$\bigcap_{g\in G} g \cdot N,$$
		we can further assume that $N$ is a polycyclic $G$-subgroup.
		Since $N$ is solvable, the derived sequence $N^{(i)}$ of $N$ gives a sequence of normal $G$-subgroups of $H$
		$$\{1_H\}=N^{(m)} \le \cdots  \le N^{(1)} \le N^{(0)}=N  \le H,$$
		and the finite generation assumption (for all subgroups of $N$) implies that the quotient abelian groups $N^{(i)}/N^{(i+1)}$ are all finitely generated. Hence (3) follows from (2). 
		
		(4) is clear by the definition of the Galois cohomology set. 
		To our best knowledge, Lesieutre \cite[Lemma 13]{Le18} is the first who explicitly mentioned (4) 
		and effectively applied (4) for the existence of a smooth projective variety with infinitely many real forms. 
	\end{proof}

	\ssec{Proof of Theorem~\ref{thm12}}
	\hfill

	In the subsection, we prove Theorem~\ref{thm12} which is restated 
	as Theorem \ref{thm221} below. Let us start from some lemmas. 
	
	\begin{lemma}\label{lem217}
		Let $f: \R^n/\Z^n\to \R^n/\Z^n$ be a Lie group automorphism 
		of order $2$.
		Let $G \cnec \langle f \rangle \le \Aut(\R^n/\Z^n)$
		act naturally on $\R^n/\Z^n$.  
		Then $H^1(G, \R^n/\Z^n)$ is finite. 
		
	\end{lemma}
	
	Here we provide two different proofs of this lemma. 
	
	\begin{proof}[First proof of Lemma \ref{lem217}]
		Since the Lie group $\R^n$ is the universal covering of $\R^n/\Z^n$, it follows that $f$ can be lifted to a Lie group automorphism $g$ of $\R^n$. 
		Note that $g$ is a linear map. In fact, since $g$ preserves addition in $\R^n$ and $g(\Z^n)=\Z^n$, it follows that $g$ is $\Q$-linear on $\Q^n$.
		Since $g$ is a diffeomorphism (in particular, continuous), we have that $g$ is $\R$-linear on $\R^n$. The restriction $g|_{\Z^n}: \Z^n\to \Z^n$ is an automorphism of the free abelian group $\Z^n$ of order at most $2$. We may and will view $\R^n$ and $\Z^n$ as $G$-groups via $g$ and $g|_{\Z^n}$ respectively. Thus we have the following exact sequence of $G$-groups
		$$0\to \Z^n\to\R^n\to\R^n/\Z^n\to 0.$$

		As these are abelian groups, hence $G$-modules, we have the following long exact sequence of cohomology groups 
		$$H^1(G,\R^n)\to H^1(G,\R^n/\Z^n)\to H^2(G, \Z^n)\to H^2(G, \R^n).$$
		By Comessatti's Lemma (see \cite[Proposition 2]{Si82}), it suffices to prove the finiteness of $H^1(G,\R^n/\Z^n)$ in the following three cases: 
		\begin{enumerate}
		\item $n=1$, $g|_{\Z}={\rm id}_{\Z}$; 
		\item $n=1$, $g|_\Z=-{\rm id}_{\Z}$;
		
		\item $n=2$, $g|_{\Z^2} (a,b)=(a+b,-b)$ for any $(a,b)\in \Z^2$.
		\end{enumerate} 
		By \cite[Chapter VI, Proposition 7.1]{HS97} and the above long exact sequence, $H^1(G,\R^n/\Z^n)$ in the three cases is $\Z/2\Z$, $0$, $0$ respectively.
	\end{proof}
	
	\begin{proof}[Second proof of Lemma \ref{lem217}]
		Since $T = \R^n/\Z^n$ is a commutative $G$-group,
		we have group isomorphisms
		$$Z^1(G,\R^n/\Z^n) \eto \Ker(f + \id_T) \subset T,$$
		and
		$$B^1(G,\R^n/\Z^n) \eto \Ima(f - \id_T) \subset T,$$
		where both maps are defined by $\gs \mapsto \gs(f)$.
		Since $\Ker(f + \id_T)$ is a Lie subgroup of $T$ and $T$ is compact,
		it has only finitely many connected components.
		Thus to show that $H^1(G, \R^n/\Z^n)$ is finite,
		it suffices to show that 
		\begin{equation}\label{eqn-dimKI}
			\dim \Ker(f + \id_T) = \dim \Ima(f - \id_T).    
		\end{equation}
		Let $T_f : \R^n \to \R^n$ be the tangent map of $f$ at the origin.
		Since $T_f^2 = \id_T$, we have
		$$\R^n = \Ker(T_f + \id_T) \oplus \Ker(T_f - \id_T).$$
		Hence
		$$\dim \Ker(T_f + \id_T) = \dim \Ima(T_f - \id_T),$$
		which implies~\eqref{eqn-dimKI}.
	\end{proof}
	
	\begin{lemma}\label{lem218}
		Let $A_\R$ be a real abelian variety and let $A=A_\R\times_{\rm Spec\, \R} {\rm Spec}\,\C$.  Then $H^1(G_{\C/\R}, A)$ is finite.
	\end{lemma}
	
	\begin{proof}

		Recall that $G_{\C/\R}$ acts on $A$ via the anti-holomorphic involution 
		$\imath \cnec {\rm id}_{A_\R}\times c$ of $A$. 
		Moreover, $\imath$ is a group homomorphism of $A$. Then as real Lie groups, we may identify $A$ with $\R^{2d}/\Z^{2d}$ 
		where $d = \dim A$, 
		and $\imath$ corresponds to a Lie group automorphism of $\R^{2d}/\Z^{2d}$ of order $2$. 
		By Lemma \ref{lem217}, $H^1(G_{\C/\R}, A)$ is finite.
	\end{proof}

	\begin{lemma}\label{lem219}
		Let $A_\R$ be a 
		connected algebraic group over $\R$ and let $A=A_\R\times_{\rm Spec\, \R} {\rm Spec}\,\C$.  Then $H^1(G_{\C/\R}, A)$ is finite.
	\end{lemma}

	\begin{proof}
		By Barsotti--Chevalley's structure theorem~\cite[Theorem 8.27, Notes 8.30]{Mi17}, 
		$A_\R$ (resp. $A$) has a unique normal connected linear algebraic subgroup $N_\R$ (resp. $N:=N_\R\times_{\rm Spec\, \R} {\rm Spec}\,\C$) such that the quotient $P_\R:=A_\R/N_\R$ (resp. $P:={P_\R}\times_{\rm Spec\, \R} {\rm Spec}\,\C$) is an abelian variety. 
		Then we have an exact sequence 
		$$H^1(G_{\C/\R}, N) \to H^1(G_{\C/\R}, A) \to H^1(G_{\C/\R}, P),$$
		as pointed sets, induced from the exact sequence of $G_{\C/\R}$-groups
		$$1 \to N \to A \to P \to 1.$$ 
		By Lemma \ref{lem218}, $H^1(G_{\C/\R}, P)$ is finite. 
		Thus, by \cite[Page 53, Corollary 3]{Se02}, 
		it suffices to show that $H^1(G_{\C/\R}, N_\phi)$ is finite for any $\phi \in Z^1(G_{\C/\R}, A)$
		(see Example~\ref{ex-RF} (2) for the definition of $N_\phi$). 
		As we mentioned in Example~\ref{ex-RF} (3),
		since $N_{\R}$ is a linear algebraic group over $\R$,
		so is the real form $N_{\R, \phi}$.
		It follows from~\cite[Page 144, Theorem 4; Page 143, Examples]{Se02} 
		that $H^1(G_{\C/\R}, N_\phi)$ is finite.
	\end{proof}
	
	For a locally compact field $k$ of characteristic $0$ and a so-called $k$-group $A$ of type (ALA),  Borel and Serre (\cite[Th\'eor\`eme 6.1]{BS64}) show that $H^1(k, A)$ is finite. For $k=\R$, the following result is in some sense a generalization of \cite[Th\'eor\`eme 6.1]{BS64}.

	\begin{theorem}\label{thm220}
		Let $A_\R$ be a locally algebraic group over $\R$ and let $A=A_\R\times_{\rm Spec\, \R} {\rm Spec}\,\C$. Let $A^0$ denote the identity component of $A$. If $H^1(G_{\C/\R}, A/A^0)$ is finite (resp. countable), then $H^1(G_{\C/\R}, A)$ is finite (resp. countable) as well. In particular, $H^1(G_{\C/\R}, A)$ is finite if 
		$A_\R$ is an algebraic group 
		over $\R$.
	\end{theorem}
	
	\begin{proof}

		We have an exact sequence 
		$$H^1(G_{\C/\R}, A^0) \to H^1(G_{\C/\R}, A) \to H^1(G_{\C/\R}, A/A^0),$$
		as pointed sets, induced from the exact sequence of $G_{\C/\R}$-groups
		$$1 \to A^0 \to A \to A/A^0 \to 1.$$
		Let $A_\R^0$ denote the identity component of $A_\R$. 
		We have $A^0=A^0_\R\times_{\rm Spec\, \R} {\rm Spec}\,\C$. 
		Since $A_\R^0$ is a connected algebraic group,
		so is the real form which underlies 
		$A^0_\phi$ for all $\phi \in Z^1(G_{\C/\R},A)$
		by Example~\ref{ex-RF}.
		Thus $H^1(G_{\C/\R}, A^0_\phi)$ is finite by Lemma~\ref{lem219}. 
		The first claim then follows from \cite[Page 53, Corollary 3]{Se02}.
		
		If $A_\R$ is an algebraic group, then $A/A^0$ is finite. Hence $H^1(G_{\C/\R}, A/A^0)$ is finite by definition, and the second claim follows from the first one.
	\end{proof}
	
	\begin{theorem}\label{thm221}
		Let $V$ be a complex projective variety with a real form. Then the number of mutually non-isomorphic real forms of $V$ is at most countable.
		If 
		\begin{equation}\label{H1-disc}
			H^1(G_{\C/\R}, {\rm Aut}(V)/{\rm Aut}^0(V))
		\end{equation}
		is finite, then $V$ has only finitely many real forms up to equivalence. 
	\end{theorem}
	
	\begin{proof} The first statement follows from Theorem~\ref{thm220}, as the group ${\rm Aut}(V)/{\rm Aut}^0(V)$, hence, the set $H^1(G_{\C/\R}, {\rm Aut}(V)/{\rm Aut}^0(V))$, is countable.
		According to Example~\ref{ex-RF},
		$\Aut(V)$ is a locally algebraic group admitting a real form,
		so~\eqref{H1-disc} makes sense, 
		and we can apply Theorem~\ref{thm220} with $A = \Aut(V)$. 
		The finiteness of~\eqref{H1-disc} then implies that
		$H^1(G_{\C/\R}, {\rm Aut(V)})$ is finite, 
		thus $V$ has only finitely many real forms by Theorem~\ref{thm21}.  
	\end{proof}
	
	\begin{remark}\label{rem220}
		Theorem~\ref{thm221} was asserted in~\cite[Corollary D.1.10]{DIK00} but only proven
		when ${\rm Aut}^0(V)$ is a linear algebraic group.
		As we believe that Theorem \ref{thm221} is fundamental, we gave a complete proof here.
	\end{remark}

	\begin{proof}[Proof of Corollary~\ref{cor-ent0}]
		
		It is clear that
		if $\Aut(V)/\Aut^0(V)$ is virtually solvable,
		then so is $\Ima(\rho)$.
		By Fujiki-Lieberman's theorem~\cite[Theorem 2.10]{Br18},
		$\Ker(\rho)$ is finite.
		As $\Ima(\rho)$ embeds into $\GL(\NS(V)/\torsion)$, 
		$\Ima(\rho)$ is virtually polycyclic
		by Malcev's theorem~\cite[Page 26, Corollary 1]{Se83}.
		It follows from Proposition~\ref{prop21} (3), then (2), 
		that $H^1(G_{\C/\R},\Aut(V)/\Aut^0(V))$ is finite.
		Thus Corollary~\ref{cor-ent0} follows from Theorem~\ref{thm12}.
	\end{proof}
	
	\begin{proof}[Proof of Corollary~\ref{cor-ent0bis}]
	
	By Corollary~\ref{cor-ent0} and Tits' alternative~\cite[Theorem 1]{T72},
	the image of
	$$\rho : \Aut(V)/\Aut^0(V) \to \GL(\NS(V)/\torsion)$$
	contains a non-abelian free group.
	This implies the first statement.
	The second statement follows from Corollary~\ref{cor-ent0}
		together with \cite[Proposition 2.6 (1)]{DLOZ22}.
	\end{proof}

	\ssec{Cone conjecture and real structures}
	\hfill
	
	Now we prove Theorem~\ref{thm13}
	mentioned in the introduction 
	by clarifying some arguments of~\cite{CF19}.
	First we prove the following finiteness result, 
	which is claimed in~\cite[Lemma 2.5]{Be17} without proof. 
	We prove it here for the sake of completeness 
	(see also \cite[Section 9]{CF19}).
	\begin{lemma}\label{lem-benzerga}
		Let $\Gamma$ be a $\Z/2\Z$-group. If the semidirect product $\Gamma\rtimes\Z/2\Z$ induced by the $\Z/2\Z$-action on $\Gamma$ contains only finitely many conjugacy classes of elements of order 2, then $H^1(\Z/2\Z,\Gamma)$ is finite.
	\end{lemma}
	
	\begin{proof}
		Here we identify the elements of $\Z/2\Z$ with $\{\overline{0},\overline{1}\}$.
		Note that conjugation by $(1_{\Gamma},\overline{1})$ makes $\Gamma\rtimes\Z/2\Z$ into a $\Z/2\Z$-group, in a way that we have the following exact sequence of $\Z/2\Z$-groups:
		\[ 1 \to \Gamma \to \Gamma \rtimes \Z/2\Z \to \Z/2\Z \to 1,  
		\]
		where the induced action on $\Z/2\Z$ is trivial.
		This induces an exact sequence of pointed sets 
		\[ \{\pm 1\} \to H^1(\Z/2\Z, \Gamma) \to H^1(\Z/2\Z, \Gamma \rtimes \Z/2\Z). 
		\] 
		By~\cite[Page 53, Corollary 3]{Se02}, it suffices to show that $H^1(\Z/2\Z, \Gamma \rtimes \Z/2\Z)$ is finite.
		
		Since $\Z/2\Z$ acts on $\Gamma \rtimes \Z/2\Z$
		by conjugation, we have
		$$H^1(\Z/2\Z, \Gamma\rtimes\Z/2\Z) \simeq
		H^1(\Z/2\Z, (\Gamma\rtimes\Z/2\Z)_{\mathrm{triv}})$$
		where $(\Gamma\rtimes\Z/2\Z)_{\mathrm{triv}}$
		is the $\Z/2\Z$-group $\Gamma\rtimes\Z/2\Z$ with the trivial $\Z/2\Z$-action. The group cohomology $H^1(\Z/2\Z, (\Gamma\rtimes\Z/2\Z)_{\mathrm{triv}})$ is in bijection with the set of elements of order 1 or 2 in $\Gamma\rtimes\Z/2\Z$ modulo conjugation, which is finite by assumption.
	\end{proof}

	Let $V$ be a smooth complex projective variety. 
	The Klein automorphism group 
	$\KAut(V)$
	of $V$, is defined as the group of 
	holomorphic and anti-holomorphic automorphisms of a scheme $V \to {\rm Spec}\, {\mathbb C}$ over ${\rm Spec}\, {\mathbb R}$ to itself. 
	If $V$ admits a real structure $\imath$, then $$\KAut(V)\simeq\Aut(V/{\mathbb C}) \rtimes \langle\imath\rangle.$$
 Since $\imath$ is an automorphism of a scheme $V$, we have 
$$\imath^* :  {\mathcal O}_V(U) \simeq {\mathcal O}_V(\imath^{-1}(U))$$
for any Zariski open subset $U \subset V$. Then for $f \in {\mathcal O}_V(U)$ and for any $x \in \imath^{-1}(U)({\mathbb C})$, we have 
$$(\imath^*f)(x) = c(f(\imath(x))) = \overline{f(\imath(x))},$$
as by definition, the value $(\imath^*f)(x) \in {\mathbb C} = {\mathcal O}_{V, x}/{\mathfrak m}_{V, x}$ is uniquely determined by the condition 
$$\imath^*f - (\imath^*f)(x) \in {\mathfrak m}_{V, x}.$$
(See for instance \cite[Section 4.2]{MO15}.) This naturally extends for the pull-back of rational functions of $V$. Let $D$ be a Cartier divisor on $V$ with local equations $(f_U, U)$. We define the Cartier divisor $\overline{D}$ on $V$ by the local equations $(\imath^*f_U, \imath^{-1}(U))$.
	Then the contravariant $\Aut(V)$-action on $\Pic(V)$ extends to
	a contravariant $\KAut(V)$-action by $\imath^*({\mathcal O}_V(D)) = {\mathcal O}_V(\overline{D})$.
	It induces a contravariant $\KAut(V)$-action on $\NS(V)$, which preserves the ample cone. Note that, by the definition of $H^0(V, {\mathcal O}_V(D))$ and $H^0(V, {\mathcal O}_V(\overline{D}))$ (as vector subspaces of the rational function field of $V$), the linear system $|{\mathcal O}_V(D)|$ is free (resp. very ample) if and only if so is $|{\mathcal O}_V(\overline{D})|$.

	Let $\Aut^*(V)$ and $\KAut^*(V)$
	denote respectively the images of $\Aut(V)$ and $\KAut(V)$
	in $\GL(\NS(V)/\torsion)$.
	We have  
	$${\rm KAut}^{*}(V) 
	= \langle {\rm Aut}^{*}(V), \imath^* \rangle. 
	$$

	\begin{proposition}\label{pro-Ksgfini}
		Let $V$ be a complex projective variety and let $\Gamma$ be a subgroup of ${\rm GL}(\NS(V)/\torsion)$
		such that $\Gamma$ contains $\Gamma\cap\Aut^*(V)$ as a finite index subgroup
		and preserves ${\rm Amp} (V)$ 
		(e.g. $\Gamma = {\rm Aut}^*(V)$ or ${\rm KAut}^*(V)$).
		Suppose that the rational hull ${\rm Nef}^{+} (V)$ 
		of the nef cone $\Nef(V)$ 
		contains a rational polyhedral cone $\Sigma$ satisfying
		$$\left( \Gamma\cap\Aut^*(V) \right) \cdot \Sigma \supset \Amp(V).$$
		Then $\Gamma$ has only finitely many 
		finite subgroups, up to conjugation under $\Gamma\cap{\rm Aut}^*(V)$. 
	\end{proposition}
	
	\begin{proof}

		Since $[\Gamma : \Gamma\cap \Aut^*(V)] < \infty$, by
		Fujiki-Lieberman's theorem 
		(see e.g.~\cite[Theorem 2.10]{Br18})
		for each $v \in \Amp(V) \cap 
		(\NS(V)/\torsion)$, the stabilizer group 
		of $v$ 
		$$\{g \in \Gamma\,|\, g(v) = v\}$$
		is a finite group. In particular, for any subset $F \subset {\rm NS} (V) \otimes_{{\mathbb Z}} {\mathbb R}$ such that 
		$$F \cap {\rm Amp} (V) \cap \left({\rm NS} (V)/{\rm torsion}\right) 
		\not= \varnothing,$$ 
		the pointwisely stabilizer group of $F$
		$$Z_{\Gamma}(F) := \{g \in \Gamma \,|\, g(v) = v, \,\, \forall\, v \in F\}$$
		is a finite group as well. 
		
		Thus, by the Siegel property \cite[Theorem 3.8]{Lo14}, for any two polyhedral cones $\Pi_1$ and $\Pi_2$ in ${\rm Nef}^{+} (V)$, which are not necessarily of maximal dimension nor of the same dimension, the set 
		$$\{ g \in \Gamma \,|\, g(\Pi_1^{\circ}) \cap \Pi_2^{\circ} \cap {\rm Amp} (V) \not= \varnothing \}$$
		is a finite set as $Z_{\Gamma}(F_i)$ in \cite[Theorem 3.8]{Lo14} is a finite group as mentioned above. 
		Here and hereafter, $\Pi^{\circ}$ is the relative interior of $\Pi$.
		
		Let $\Delta$ be the set of all faces of $\Sigma$. Here $\Sigma$ itself is also considered as a face as did in \cite[Section 1]{Lo14}. Since $\Sigma$ is a rational polyhedral cone, $\Delta$ is a finite set. Hence
		$${\mathcal S} := \{ g \in \Gamma \,|\, g(\Pi_i^{\circ}) \cap \Pi_i^{\circ} \cap {\rm Amp} (V) \not= \varnothing\,\, \mbox{for some} \, \Pi_i \in \Delta \}$$
		is also a finite set. 
		
		Let $H \subset \Gamma$ be a finite subgroup. Choose $v \in {\rm Amp} (V) \cap (\NS(V)/\torsion)$. Then 
		$$v_H := \sum_{g \in H}g(v) \in {\rm Amp} (V) \cap (\NS(V)/\torsion)$$
		as $\Gamma$ preserves $\Amp(V)$ and $\NS(V)/\torsion$. 
		Since 
		$\left( \Gamma\cap\Aut^*(V) \right) \cdot \Sigma \supset \Amp(V)$, 
		there is then an element $a \in \Gamma\cap{\rm Aut}^*(V)$ such that 
		$$u_H := a(v_H) \in \Sigma \cap {\rm Amp} (V) 
		\cap (\NS(V)/\torsion).$$
		As $g(v_H) = v_H$ whenever $g \in H$, it follows that 
		$$a \circ g \circ a^{-1} (u_H) = a \circ g (v_H) = a(v_H) = u_H$$
		for all $g \in H$. Hence, considering the (unique) face $\Pi$ of $\Sigma$ such that $u_H \in \Pi^{\circ}$, we deduce that
		$$a \circ H \circ a^{-1} \subset {\mathcal S}.$$
		Since ${\mathcal S}$ is a finite set, 
		it contains only finitely many finite subgroups of $\Gamma$. 
		Thus finite subgroups of $\Gamma$ are at most finite up to conjugation under $\Gamma\cap{\rm Aut}^*(V)$.
	\end{proof}

	\begin{proof}[Proof of Theorem~\ref{thm13}]
		
		We may and will assume that $V$ has a real structure $\imath$. 
		By Theorem~\ref{thm221}, it suffices to show that 
		$H^1(G_{\C/\R}, {\rm Aut}(V)/{\rm Aut}^0 (V))$ is finite. 
		Recall that we have an exact sequence of $G_{\C/\R}$-groups  
		$$1 \to N \to {\rm Aut}(V)/{\rm Aut}^0 (V) \to {\rm Aut}^*(V) \to 1$$
		for some finite $G_{\C/\R}$-group $N$ by Fujiki-Lieberman's theorem. 
		It follows that
		$H^1(G_{\C/\R}, N_\phi)$
		is finite for all 
		$\phi \in  Z^1(G_{\C/\R},\Aut(V)/\Aut^0(V))$. By~\cite[Page 53, Corollary 3]{Se02}, 
		it suffices to show that $H^1(G_{\C/\R},\Aut^*(V)) = H^1(\langle \imath^* \rangle,\Aut^*(V))$ 
		is finite.

		First we assume that
		$\imath^* \in \Aut^*(V)$.
		Then $\KAut^*(V) = \Aut^*(V)$.
		Since the $\imath^*$-action on $\Aut^*(V)$ 
		is the conjugation by $\imath^*$,
		the set
		$H^1(\langle \imath^* \rangle,\Aut^*(V))$ is in bijection with
		the set of conjugacy classes of involutions of
		$\Aut^*(V) = \KAut^*(V)$,
		which is finite by Proposition~\ref{pro-Ksgfini}.
		Now assume that $\imath^* \not\in \Aut^*(V)$,
		then $\Aut^*(V) \rtimes \langle \imath^* \rangle = \KAut^*(V)$,
		and it follows from again Proposition~\ref{pro-Ksgfini}, together with Lemma~\ref{lem-benzerga}, that
		$H^1(\langle \imath^* \rangle,\Aut^*(V))$ is finite.
	\end{proof}
	
	\begin{remark}\label{rem22} 

		The argument \cite[Section 9, Proof of Theorem 1.1]{CF19} is correct modulo the proof of \cite[Proposition 7.4]{CF19}, which is crucial. For instance, in the proof of \cite[Proposition 7.4]{CF19}, it is unclear in general if $\{g^*(\Sigma)\}_{g^* \in {\rm Aut}(V)^{*}}$ form a fan or not. Therefore, it is in general unclear if $g^*(\Sigma) \cap \Sigma$ is a face of both $\Sigma$ and $g^*(\Sigma)$ or not, either. Even if this would be the case, it is yet unclear if the one-dimensional ray $R$ of both $\Sigma$ and $g^*(\Sigma)$ in the proof of \cite[Proposition 7.4]{CF19} is {\it inside} ${\rm Amp} (V)$ or not. Indeed, if $R$ is on the boundary of ${\rm Amp} (V)$, then the set of $g^* \in {\rm Aut}(V)^{*}$ such that 
		$$R \subset \Sigma \cap g^*(\Sigma)$$
		could be an infinite set. For instance, this is the case where $g$ is an element of the Mordell-Weil group of an elliptic K3 surface $V \to {\mathbb P}^1$ of infinite order. For this reason and the importance of Theorem~\ref{thm13}, we gave a complete proof under a slightly more general setting,
		while respecting their original arguments as much as we can. 
	\end{remark}

	\section{Proof of Theorem~\ref{thmRFsurf}}\label{sct3}

	We will prove Theorem~\ref{thmRFsurf} at the end of this section. 
	Let us begin with the following corollary of Theorem~\ref{thm12},
	originally proven by Silhol~\cite[Proposition 7]{Si82}.
	
	\begin{corollary}\label{cor-ab}
		Let $A$ be an abelian variety. Then $A$, as a complex variety, has at most finitely many non-isomorphic real forms. 
	\end{corollary}
	
	\begin{proof} 
		The proof of~\cite[Proposition 7]{Si82} is more precise,
		in that it enumerates the number of real forms.
		Here we only show the finiteness. 
		Since the $G_{\C/\R}$-group
		$${\rm Aut}(A)/{\rm Aut}^0 (A)$$
		is arithmetic~\cite[Exemples 3.5]{BS64},
		$$H^1(G_{\C/\R}, {\rm Aut}(A)/{\rm Aut}^0 (A))$$ 
		is finite by Proposition \ref{prop21} (1). 
		Thus the result follows from Theorem \ref{thm12}. 
	\end{proof}
	
	\begin{proposition}\label{prop32}
		Let $V$ be a smooth complex projective variety. 
		Assume that $\kappa (V) \ge \dim (V) -1$. 
		Then every automorphism of $V$ has zero entropy.
		As a consequence, $V$ has at most finitely many non-isomorphic real forms.  
	\end{proposition}
	
	\begin{proof} 
	The first statement is well-known. 
	Here we provide a proof for reader's convenience.
	Consider the pluricanonical map 
		$$\Phi := \Phi_{|mK_V|} : V \dasharrow B.$$ 
		Let $f \in \Aut(V)$ be an automorphism of $V$. By the finiteness of the pluricanonical representation \cite[Theorem 14.10]{Ue75}, the action $f_{\tilde{B}}$ of $f$ on an equivariant resolution $\tilde{B}$ of $B$ is finite. Thus, all the dynamical degrees of $f_{\tilde{B}}$ equal $1$. Since a general fiber of $\Phi$ is of dimension at most $1$, the relative dynamical degrees of $f$ are also $1$. Hence the first dynamical degree of $f$ is 1 and $f$ has zero entropy by the product formula (\cite[Theorem 1.1]{DN11} or \cite{Tr16}). Proposition~\ref{prop32} then follows from Corollary~\ref{cor-ent0}.
	\end{proof}

	Recall that a minimal surface $S$ with $\kappa (S) = 0$ is either a K3 surface, 
	an Enriques surface, an abelian surface or a hyperelliptic surface. 
	Recall also that an irrational surface $S$ with $\kappa (S) = -\infty$ admits a genus $0$ fibration $\pi : S \to B$, which is nothing but the Albanese morphism of $S$, over a smooth projective curve $B$ of genus $g(B) \ge 1$.

	\begin{proposition}\label{pro-IRHE}
		Let $S$ be a smooth complex projective surface birational to an irrational ruled surface
		or a hyperelliptic surface. 
		Then every automorphism of $S$ has zero entropy.
		As a consequence, $S$ has at most finitely many non-isomorphic real forms.
	\end{proposition}
	
	The first statement of Proposition~\ref{pro-IRHE}
	is also well-known; 
	see~\cite[Proposition 1]{Ca99} for a more general statement. As the proof is simple, we include it here for reader's convenience.
	
	\begin{proof} 

		Let $S \to B$ be the Albanese morphism, which is a fibration with $\dim B = 1$ in each case. By the universal property, every automorphism of $S$ preserves this fibration. Since the base and general fibers of the fibration are curves, by the product formula (\cite[Theorem 1.1]{DN11} or \cite{Tr16}), every automorphism of $S$ has zero entropy. Proposition~\ref{pro-IRHE} then follows from Corollary~\ref{cor-ent0}. 
	\end{proof}

	\begin{proposition}\label{prop34}
		Let $S$ be a smooth complex projective surface which is birational to an abelian surface $A$. 
		Then $S$ has at most finitely many non-isomorphic real forms.
	\end{proposition}
	
	\begin{proof} 
	
		It suffices by Theorem~\ref{thm21} 
		to show that
		$H^1(G_{\C/\R}, \Aut(S))$ is finite.
		
		By running the minimal model program, 
		$S$ is obtained by a sequence of blow-ups 
		$$\pi : S = S_k \to \cdots \to  S_1 \to S_0 = A$$ at $k \ge 0$ reduced points.
		If $k = 0$, then Proposition~\ref{prop34}
		is contained in Corollary~\ref{cor-ab}.
		Suppose that $k = 1$, then we can choose the origin of $A$ to be the blow-up center 
		$o$ of $\pi : S \to A$, 
		and 
		$${\rm Aut}(S) \simeq {\rm Dec} (A, o) = {\rm Aut}_{{\rm group}}(A).$$ 
		Since ${\rm Aut}_{{\rm group}}(A)$ is an arithmetic $G_{\C/\R}$-group,
		$H^1(G_{\C/\R}, \Aut(S))$ is finite by Proposition~\ref{prop21} (1).

		Now assume that $k \ge 2$.
		Let $E_1,\ldots,E_k$ be the irreducible components of the exceptional set of $\pi$. 
		Then 
		$$H := {\rm Dec}(S, E_1, \ldots, E_k)$$
		is a finite index subgroup of $\Aut(S)$ and $H$ descends to a subgroup of $\Dec(A, \gS)$. Here $\gS \subset A$ is the blow-up center of $S_2 \to A$, which is a subscheme of length $2$, and $\Dec(A, \gS)$ is the decomposition group of the \textit{closed subscheme} $\Sigma \subset A$.
		We choose a point $o$ in the support of $\gS$ as the origin of $A$.
		
		\medskip
		{\bf \noindent Case 1: $\gS$ is supported at one point $o \in A$.}

		In this case, we have, for some $v \in T_{A,o}$, 
		$$\Dec(A, \gS) = \Set{f \in \Aut_{{\rm group}}(A) | 
		[(df)_o(v)] = [v] \in \P(T_{A,o})}.$$
		
		\begin{claim}\label{cl22}
			$\Dec(A, \gS)$ is a solvable group.
		\end{claim}
		
		\begin{proof}
			By assumption, there is a ${\mathbb C}$-basis $\langle v, u \rangle$ of $T_{A, o}$ such that the action of $f \in \Dec(A, \gS)$ on the tangent space $T_{A, o}$ is of the form
			$$(df)_o = \begin{pmatrix} 
				c(f) & a(f) \\ 
				0& b(f)\\ 
			\end{pmatrix}\,\, (c(f), b(f) \in {\mathbb C}^{\times}, a(f) \in {\mathbb C})$$
			with respect to the basis $\langle v, u \rangle$. 
			Thus $\Dec(A, \gS)$ is solvable, as the representation 
			$$\Aut_{{\rm group}}(A) = {\rm Dec}(A,o) \to {\rm GL}(T_{A,o}), \ \ f \mapsto (df)_o$$ 
			is faithful.
		\end{proof}
		
		Consider the natural faithful representation
		$$\rho : \Dec(A, \gS) \subset \Aut_{\rm group}(A) \hto \GL(H^1(A,\Z)).$$  
		Since $\Dec(A, \gS)$ is solvable by Claim~\ref{cl22}, 
		and since $H^1(A,\Z)$ is a free abelian group of finite rank, $\Dec(A, \gS)$ is then a polycyclic group by Malcev's theorem~\cite[Page 26, Corollary 1]{Se83}.
		It follows that $H$ is polycyclic as well, and $\Aut(S)$ is virtually polycyclic.
		Thus $H^1(G_{\C/\R}, \Aut (S))$ is finite by Proposition~\ref{prop21} (3). 
		
		\medskip
		{\bf \noindent Case 2: $\gS$ is supported at two points $o,P \in A$ such that $P$ is not torsion.}
		
		Let $B$ be the irreducible component of the Zariski closure of 
		$\{nP \,|\, n \in \Z\}$
		containing the origin $o$:
		$$o \in B \subset \overline{\{nP \,|\, n \in \Z\}}^{{\rm Zar}}.$$ 
		Since $P$ is not a torsion point,
		$B$ is either an elliptic curve $E$ (with the origin $o$) or $A$. 
		\begin{claim}\label{cl21}
			$\Dec(A,o,P)$ is a finite group.
		\end{claim}
		\begin{proof} 
			Since $\Dec (A, o, P)$ acts trivially on $\{nP \,|\, n \in \Z\}$,  
			and therefore on $B$, the result follows if $B = A$. 
			Consider the case where $B = E$. Consider the elliptic curve $C := A/E$ and the quotient morphism $p : A \to C$. We choose $p(o) \in C$ as the origin of the elliptic curve $C$. 
			Then $\Dec(A,o, P)$ embeds into $\Aut_{\group}(C)$. 
			Since $C$ is an elliptic curve, the group $\Aut_{\group}(C)$ is finite. 
			Thus the result follows also in the case where $B=E$.   
		\end{proof}
		
		Recall that $H \subset \Dec(A, o, P)$ and $H$ is a finite index subgroup of $\Aut(S)$,
		Claim~\ref{cl21} implies that $\Aut(S)$ is finite, 
		hence $H^1(G_{\C/\R}, \Aut (S))$ is finite.
		
		\medskip
		{\bf \noindent Case 3: $\gS$ is supported at two points $o,P \in A$ such that $P$ is torsion.}
		
		This is the last case we need to consider.  
		Thanks to the first two cases,
		up to rearranging the blow-up sequence,
		we can reduce to the case where 
		$S \to A$ is the blow-up at finitely many distinct torsion points, 
		including the origin $o$, of $A$. Then 
		$${\rm Ine}_{{\rm group}}(A, A[N]) \subset H \subset {\rm Dec}_{{\rm group}}(A, A[N]) = {\rm Aut}_{{\rm group}}(A)$$
		for some $N > 0$, where $A[N] \simeq ({\mathbb Z}/N)^4$ is the subgroup of torsion points of order dividing $N$. 
		Here we note that $A[N]$ is preserved by ${\rm Aut}_{{\rm group}}(A)$ and 
		$$[{\rm Dec}_{{\rm group}}(A, A[N]) : {\rm Ine}_{{\rm group}}(A, A[N])] < \infty.$$ 
		Since $\Aut_{{\rm group}}(A)$ is arithmetic, it follows that $H$ and hence $\Aut(S)$ are also arithmetic. Therefore, by Proposition \ref{prop21} (1), $H^1(G_{\C/\R}, {\rm Aut}(S))$ is a finite set. Hence $S$ has at most finitely many real forms by Theorem \ref{thm21}.
	\end{proof}

	\begin{proof}[Proof of Theorem~\ref{thmRFsurf}] Let $S$ be a smooth complex projective surface with infinitely many mutually non-isomorphic real forms. We may assume that $S$ is not rational. Then by Propositions \ref{prop32}, \ref{pro-IRHE} and \ref{prop34}, $S$ is birational to a K3 surface or an Enriques surface. 
	
	Suppose that $S$ is minimal. Then $S$ is a K3 surface or an Enriques surface. By \cite[Theorem 2.1]{Ka97} (see also \cite{St85} and \cite{Na85}), the cone conjecture holds for $S$, that is, there exists a rational polyhedral fundamental domain for the action of $\Aut^{*}(S)$ on the cone $\Nef^{+}(S)$. By Theorem \ref{thm13}, $S$ has at most finitely many non-isomorphic real forms. This is a contradiction and therefore, $S$ is non-minimal. 
	\end{proof}
	
	\begin{remark}\label{rem31}
	Let $S$ be a smooth projective surface. Then the group $\Aut(S)/\Aut^0(S)$ is finitely generated unless $S$ is either rational or non-minimal and birational to an abelian surface, a K3 surface or an Enriques surface.
    Indeed, our proof of Theorem~\ref{thmRFsurf} shows that the group $\Aut(S)/\Aut^0(S)$ is either a polycyclic group or an arithmetic group, up to finite kernel and cokernel, or satisfies the cone conjecture. In the first two cases $\Aut(S)/\Aut^0(S)$ is clearly finitely generated. In the last case one can deduce from \cite[Corollary 4.15]{Lo14} that $\Aut(S)/\Aut^0(S)$ is finitely generated as well. It would be interesting to study relations between finiteness of real forms and finite generation of the group $\Aut(S)/\Aut^0(S)$ more closely.
	\end{remark}

	\section{Kummer surfaces of product type}\label{sec-Kummer}
	
	Throughout this section, let
	$\bk$ be a field of characteristics zero 
	(e.g. $\bk = \ol{\Q}$, $\R$, or $\C$).
	
	\ssec{Kummer surfaces of product type
	and their double Kummer pencils}\label{ssec-Kummer}

	Let $E$ and $F$ be the projective elliptic curves over $\bk$
	given by the affine Weierstrass equation
	\begin{equation}\label{eq31}
		y^2 = x(x-1)(x-s),
	\end{equation}
	\begin{equation}\label{eq32}
		y'^2 = x'(x'-1)(x'-t)
	\end{equation}
	for some $s,t \in \bk \setminus \Set{0,1}$
	respectively.
	Note that $E/\langle -1_E \rangle = \P^1$, the associated quotient map $E \to \P^1$ is given by $(x,y)\mapsto x$, 
	and the points $0$, $1$, $t$ and $\infty$ of $\P^1$ are exactly the branch points of this quotient map. The same holds for $F$ if we replace $s$ by $t$.
	Let
	$$ \tau_0, \tau_1,\tau_2,\tau_3 \in E  ; \ \ \ \tau_0',\tau_1',\tau_2',\tau_3' \in F$$
	be the pre-images of 
	$$0,1,s,\infty \in \P^1; \ \ \  0,1,t,\infty \in \P^1 $$
under the double covers $E \to \P^1$, $F \to \P^1$ respectively.
We set $\tau_0$ and $\tau_0'$ to be the origins of $E$ and $F$ respectively;
the points $\tau_i \in E, \tau'_i \in F$ are thus $2$-torsion.

	Let
	$$X := {\rm Km} (E \times F)$$
	be the Kummer K3 surface associated to the product abelian surface $E \times F$, that is, the minimal resolution of the quotient surface $E \times F/\langle -1_{E\times F} \rangle$. 
	Then $X$ contains 24 smooth $(-2)$-curves, 
	which form the so-called double Kummer pencil on $X$, as in Figure~\ref{fig1}. 
	Here the smooth rational curves $E_i$, $F_i$ ($0 \le i \le 3$) arise from the elliptic curves $E \times \{\tau'_i\}$, $\{\tau_i\} \times F$ on $E \times F$, 
	and $C_{ij}$ ($0\le i,j \le 3$) are the exceptional curves over the $A_1$-singularities of the quotient surface $E \times F/\langle -1_{E\times F} \rangle$. 
	Each of these $24$ curves is defined over $\bk$,
	as well as the points 
	$$P_{ij} \cnec E_i \cap C_{ij} \ \text{ and } \ P_{ij}' \cnec F_j \cap C_{ij}.$$

	\begin{figure}
		\unitlength 0.1in
		\begin{picture}(25.000000,24.000000)(-1.000000,-23.500000)
			\put(4.500000, -22.000000){\makebox(0,0)[rb]{$E_0$}}%
			\put(9.500000, -22.000000){\makebox(0,0)[rb]{$E_1$}}%
			\put(14.500000, -22.000000){\makebox(0,0)[rb]{$E_2$}}%
			\put(19.500000, -22.000000){\makebox(0,0)[rb]{$E_3$}}%
			\put(0.250000, -18.500000){\makebox(0,0)[lb]{$F_0$}}%
			\put(0.250000, -13.500000){\makebox(0,0)[lb]{$F_1$}}%
			\put(0.250000, -8.500000){\makebox(0,0)[lb]{$F_2$}}%
			\put(0.250000, -3.500000){\makebox(0,0)[lb]{$F_3$}}%
			\put(6.000000, -16.000000){\makebox(0,0)[lt]{$C_{00}$}}%
			\put(6.000000, -11.000000){\makebox(0,0)[lt]{$C_{01}$}}%
			\put(6.000000, -6.000000){\makebox(0,0)[lt]{$C_{02}$}}%
			\put(6.000000, -1.000000){\makebox(0,0)[lt]{$C_{03}$}}%
			\put(11.000000, -16.000000){\makebox(0,0)[lt]{$C_{10}$}}%
			\put(11.000000, -11.000000){\makebox(0,0)[lt]{$C_{11}$}}%
			\put(11.000000, -6.000000){\makebox(0,0)[lt]{$C_{12}$}}%
			\put(11.000000, -1.000000){\makebox(0,0)[lt]{$C_{13}$}}%
			\put(16.000000, -16.000000){\makebox(0,0)[lt]{$C_{20}$}}%
			\put(16.000000, -11.000000){\makebox(0,0)[lt]{$C_{21}$}}%
			\put(16.000000, -6.000000){\makebox(0,0)[lt]{$C_{22}$}}%
			\put(16.000000, -1.000000){\makebox(0,0)[lt]{$C_{23}$}}%
			\put(21.000000, -16.000000){\makebox(0,0)[lt]{$C_{30}$}}%
			\put(21.000000, -11.000000){\makebox(0,0)[lt]{$C_{31}$}}%
			\put(21.000000, -6.000000){\makebox(0,0)[lt]{$C_{32}$}}%
			\put(21.000000, -1.000000){\makebox(0,0)[lt]{$C_{33}$}}%
			\special{pa 500 2200}%
			\special{pa 500 0}%
			\special{fp}%
			\special{pa 1000 2200}%
			\special{pa 1000 0}%
			\special{fp}%
			\special{pa 1500 2200}%
			\special{pa 1500 0}%
			\special{fp}%
			\special{pa 2000 2200}%
			\special{pa 2000 0}%
			\special{fp}%
			\special{pa 0 1900}%
			\special{pa 450 1900}%
			\special{fp}%
			\special{pa 550 1900}%
			\special{pa 950 1900}%
			\special{fp}%
			\special{pa 1050 1900}%
			\special{pa 1450 1900}%
			\special{fp}%
			\special{pa 1550 1900}%
			\special{pa 1950 1900}%
			\special{fp}%
			\special{pa 0 1400}%
			\special{pa 450 1400}%
			\special{fp}%
			\special{pa 550 1400}%
			\special{pa 950 1400}%
			\special{fp}%
			\special{pa 1050 1400}%
			\special{pa 1450 1400}%
			\special{fp}%
			\special{pa 1550 1400}%
			\special{pa 1950 1400}%
			\special{fp}%
			\special{pa 0 900}%
			\special{pa 450 900}%
			\special{fp}%
			\special{pa 550 900}%
			\special{pa 950 900}%
			\special{fp}%
			\special{pa 1050 900}%
			\special{pa 1450 900}%
			\special{fp}%
			\special{pa 1550 900}%
			\special{pa 1950 900}%
			\special{fp}%
			\special{pa 0 400}%
			\special{pa 450 400}%
			\special{fp}%
			\special{pa 550 400}%
			\special{pa 950 400}%
			\special{fp}%
			\special{pa 1050 400}%
			\special{pa 1450 400}%
			\special{fp}%
			\special{pa 1550 400}%
			\special{pa 1950 400}%
			\special{fp}%
			\special{pa 200 2000}%
			\special{pa 600 1600}%
			\special{fp}%
			\special{pa 200 1500}%
			\special{pa 600 1100}%
			\special{fp}%
			\special{pa 200 1000}%
			\special{pa 600 600}%
			\special{fp}%
			\special{pa 200 500}%
			\special{pa 600 100}%
			\special{fp}%
			\special{pa 700 2000}%
			\special{pa 1100 1600}%
			\special{fp}%
			\special{pa 700 1500}%
			\special{pa 1100 1100}%
			\special{fp}%
			\special{pa 700 1000}%
			\special{pa 1100 600}%
			\special{fp}%
			\special{pa 700 500}%
			\special{pa 1100 100}%
			\special{fp}%
			\special{pa 1200 2000}%
			\special{pa 1600 1600}%
			\special{fp}%
			\special{pa 1200 1500}%
			\special{pa 1600 1100}%
			\special{fp}%
			\special{pa 1200 1000}%
			\special{pa 1600 600}%
			\special{fp}%
			\special{pa 1200 500}%
			\special{pa 1600 100}%
			\special{fp}%
			\special{pa 1700 2000}%
			\special{pa 2100 1600}%
			\special{fp}%
			\special{pa 1700 1500}%
			\special{pa 2100 1100}%
			\special{fp}%
			\special{pa 1700 1000}%
			\special{pa 2100 600}%
			\special{fp}%
			\special{pa 1700 500}%
			\special{pa 2100 100}%
			\special{fp}%
		\end{picture}%
		\caption{Curves $E_i$, $F_j$ and $C_{ij}$}
		\label{fig1}
	\end{figure}
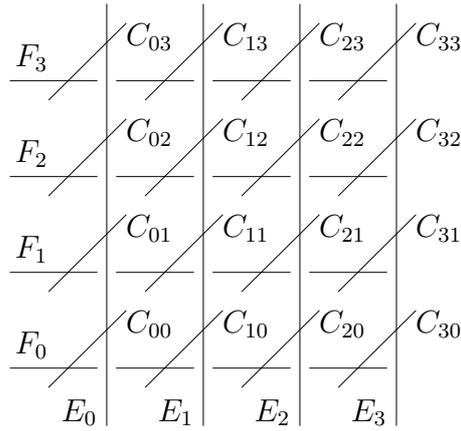

	We can use the same $x$ (resp. $x'$) in the defining equations of $E$ and $F$ 
	to denote the induced
	affine coordinates of $E_i$ and $F_j$, so that
	\begin{equation}\label{eq33}
		x(P_{i0}) = 0, \ \ x(P_{i1}) = 1, \ \  x(P_{i2}) = s, \ \ x(P_{i3}) = \infty
	\end{equation}
	on $E_i$ with respect to the coordinate $x$ and
	\begin{equation}\label{eq34}
		x'(P_{0j}') = 0, \ \  x'(P_{1j}') = 1, \ \ x'(P_{2j}') = t, \ \ x'(P_{3j}') = \infty
	\end{equation}
	on $F_j$ with respect to the coordinate $x'$.
	
	Note that the coordinate values of points are {\it different} from the ones in \cite{DO19} and \cite{DOY23} as we found that the current ones are more convenient to study the Enriques surface $Z$ defined in the next subsection, whereas the previous ones were more convenient to study the rational surface $T$ there.

\ssec{First Jacobian fibration}\label{subsec:first_jac}
	
	From now on
	until the end of Section~\ref{sec-Kummer}, 
	we assume that $\bk$ is
	algebraically closed (of characteristic zero).

	Let $D_1$ be the divisor on $X$ defined by 
	$$D_1 := F_0 + C_{10} + E_1 + C_{13} + F_3 + C_{23} + E_2 + C_{20};$$
	see Figure \ref{fig2}.
	Since $D_1$ is nef and $D_1^2 = 0$, 
	it defines an elliptic fibration
	$${\Phi}_{|D_1|} : X \to B_1 := {\mathbb P}^1,$$
	and $D_1$ is a fiber as it is reduced and connected
	(see e.g.~\cite[Proposition 2.3.10]{Hu16}).
	Define also
	$$
	D_1' := E_0 + C_{01} + F_1 + C_{31} + E_3 + C_{32} + F_2 + C_{02};$$
	see Figure \ref{fig2}.
	As $D_1'$ is reduced and connected, and satisfies
	$D_1'^2 = 0$ and $D_1 \cdot D'_1 = 0$,
	necessarily $D_1'$ is also a fiber of ${\Phi}_{|D_1|}$ 
	by the Hodge index theorem.
	Note that a smooth rational curve $C$ on $X$ is a section of $\Phi_{|D_1|}$ if and only if $C\cdot D_1 = 1$. In particular, ${\Phi}_{|D_1|}$ has sections $C_{21}$, $C_{12}$, $C_{03}$ and $C_{30}$.

	\begin{figure}
		\unitlength 0.1in
		\begin{picture}(25.000000,24.000000)(-1.000000,-23.500000)
		
			\put(-10.500000, -22.000000){\makebox(0,0)[rb]{$E_0$}}%
			\put(-5.500000, -22.000000){\makebox(0,0)[rb]{$E_1$}}%
			\put(-0.500000, -22.000000){\makebox(0,0)[rb]{$E_2$}}%
			\put(4.500000, -22.000000){\makebox(0,0)[rb]{$E_3$}}%
			\put(-14.750000, -18.500000){\makebox(0,0)[lb]{$F_0$}}%
			\put(-14.750000, -13.500000){\makebox(0,0)[lb]{$F_1$}}%
			\put(-14.750000, -8.500000){\makebox(0,0)[lb]{$F_2$}}%
			\put(-14.750000, -3.500000){\makebox(0,0)[lb]{$F_3$}}%
			\put(-9.000000, -16.000000){\makebox(0,0)[lt]{$C_{00}$}}%
			\put(-9.000000, -11.000000){\makebox(0,0)[lt]{$C_{01}$}}%
			\put(-9.000000, -6.000000){\makebox(0,0)[lt]{$C_{02}$}}%
			\put(-9.000000, -1.000000){\makebox(0,0)[lt]{$C_{03}$}}%
			\put(-4.000000, -16.000000){\makebox(0,0)[lt]{$C_{10}$}}%
			\put(-4.000000, -11.000000){\makebox(0,0)[lt]{$C_{11}$}}%
			\put(-4.000000, -6.000000){\makebox(0,0)[lt]{$C_{12}$}}%
			\put(-4.000000, -1.000000){\makebox(0,0)[lt]{$C_{13}$}}%
			\put(1.000000, -16.000000){\makebox(0,0)[lt]{$C_{20}$}}%
			\put(1.000000, -11.000000){\makebox(0,0)[lt]{$C_{21}$}}%
			\put(1.000000, -6.000000){\makebox(0,0)[lt]{$C_{22}$}}%
			\put(1.000000, -1.000000){\makebox(0,0)[lt]{$C_{23}$}}%
			\put(6.000000, -16.000000){\makebox(0,0)[lt]{$C_{30}$}}%
			\put(6.000000, -11.000000){\makebox(0,0)[lt]{$C_{31}$}}%
			\put(6.000000, -6.000000){\makebox(0,0)[lt]{$C_{32}$}}%
			\put(6.000000, -1.000000){\makebox(0,0)[lt]{$C_{33}$}}%
			
			\linethickness{0.1mm}
			\put(-10.000000, -22.000000){\line(0,1){22}}
                          \linethickness{1mm}
			 \put(-5.000000, -22.000000){\line(0,1){22}}
                          \linethickness{1mm}
			 \put(0.000000, -22.000000){\line(0,1){22}}
			 \linethickness{0.1mm}
			 \put(5.000000, -22.000000){\line(0,1){22}}
                          \linethickness{1mm}
			 \put(-15.000000, -19.000000){\line(1,0){4.5}}
			\linethickness{1mm}
			\put(-9.500000, -19.000000){\line(1,0){4}}
                         \linethickness{1mm}
			\put(-4.500000, -19.000000){\line(1,0){4}}
                         \linethickness{1mm}
			\put(0.500000, -19.000000){\line(1,0){4}}
			 \linethickness{0.1mm}
			 \put(-15.000000, -14.000000){\line(1,0){4.5}}
			 \linethickness{0.1mm}
			 \put(-9.500000, -14.000000){\line(1,0){4}}
			 \linethickness{0.1mm}
			 \put(-4.500000, -14.000000){\line(1,0){4}}
			\linethickness{0.1mm}
			 \put(0.500000, -14.000000){\line(1,0){4}}
			 \linethickness{0.1mm}
			 \put(-15.000000, -9.000000){\line(1,0){4.5}}
			 \linethickness{0.1mm}
			 \put(-9.500000, -9.000000){\line(1,0){4}}
			 \linethickness{0.1mm}
			 \put(-4.500000, -9.000000){\line(1,0){4}}
			 \linethickness{0.1mm}
			 \put(0.500000, -9.000000){\line(1,0){4}}
			 \linethickness{1mm}
			 \put(-15.000000, -4.000000){\line(1,0){4.5}}
			 \linethickness{1mm}
			 \put(-9.500000, -4.000000){\line(1,0){4}}
			 \linethickness{1mm}
			 \put(-4.500000, -4.000000){\line(1,0){4}}
			 \linethickness{1mm}
			 \put(0.500000, -4.000000){\line(1,0){4}}
			                
                \put(19.500000, -22.000000){\makebox(0,0)[rb]{$E_0$}}%
			\put(24.500000, -22.000000){\makebox(0,0)[rb]{$E_1$}}%
			\put(29.500000, -22.000000){\makebox(0,0)[rb]{$E_2$}}%
			\put(34.500000, -22.000000){\makebox(0,0)[rb]{$E_3$}}%
			\put(15.250000, -18.500000){\makebox(0,0)[lb]{$F_0$}}%
			\put(15.250000, -13.500000){\makebox(0,0)[lb]{$F_1$}}%
			\put(15.250000, -8.500000){\makebox(0,0)[lb]{$F_2$}}%
			\put(15.250000, -3.500000){\makebox(0,0)[lb]{$F_3$}}%
			\put(21.000000, -16.000000){\makebox(0,0)[lt]{$C_{00}$}}%
			\put(21.000000, -11.000000){\makebox(0,0)[lt]{$C_{01}$}}%
			\put(21.000000, -6.000000){\makebox(0,0)[lt]{$C_{02}$}}%
			\put(21.000000, -1.000000){\makebox(0,0)[lt]{$C_{03}$}}%
			\put(26.000000, -16.000000){\makebox(0,0)[lt]{$C_{10}$}}%
			\put(26.000000, -11.000000){\makebox(0,0)[lt]{$C_{11}$}}%
			\put(26.000000, -6.000000){\makebox(0,0)[lt]{$C_{12}$}}%
			\put(26.000000, -1.000000){\makebox(0,0)[lt]{$C_{13}$}}%
			\put(31.000000, -16.000000){\makebox(0,0)[lt]{$C_{20}$}}%
			\put(31.000000, -11.000000){\makebox(0,0)[lt]{$C_{21}$}}%
			\put(31.000000, -6.000000){\makebox(0,0)[lt]{$C_{22}$}}%
			\put(31.000000, -1.000000){\makebox(0,0)[lt]{$C_{23}$}}%
			\put(36.000000, -16.000000){\makebox(0,0)[lt]{$C_{30}$}}%
			\put(36.000000, -11.000000){\makebox(0,0)[lt]{$C_{31}$}}%
			\put(36.000000, -6.000000){\makebox(0,0)[lt]{$C_{32}$}}%
			\put(36.000000, -1.000000){\makebox(0,0)[lt]{$C_{33}$}}%
			
			\linethickness{1mm}
			\put(20.000000, -22.000000){\line(0,1){22}}
                          \linethickness{0.1mm}
			 \put(25.000000, -22.000000){\line(0,1){22}}
                          \linethickness{0.1mm}
			 \put(30.000000, -22.000000){\line(0,1){22}}
			 \linethickness{1mm}
			 \put(35.000000, -22.000000){\line(0,1){22}}
                          \linethickness{0.1mm}
			 \put(15.000000, -19.000000){\line(1,0){4.5}}
			\linethickness{0.1mm}
			\put(20.500000, -19.000000){\line(1,0){4}}
                         \linethickness{0.1mm}
			\put(25.500000, -19.000000){\line(1,0){4}}
                         \linethickness{0.1mm}
			\put(30.500000, -19.000000){\line(1,0){4}}
			 \linethickness{1mm}
			 \put(15.000000, -14.000000){\line(1,0){4.5}}
			 \linethickness{1mm}
			 \put(20.500000, -14.000000){\line(1,0){4}}
			 \linethickness{1mm}
			 \put(25.500000, -14.000000){\line(1,0){4}}
			\linethickness{1mm}
			 \put(30.500000, -14.000000){\line(1,0){4}}
			 \linethickness{1mm}
			 \put(15.000000, -9.000000){\line(1,0){4.5}}
			 \linethickness{1mm}
			 \put(20.500000, -9.000000){\line(1,0){4}}
			 \linethickness{1mm}
			 \put(25.500000, -9.000000){\line(1,0){4}}
			 \linethickness{1mm}
			 \put(30.500000, -9.000000){\line(1,0){4}}
			 \linethickness{0.1mm}
			 \put(15.000000, -4.000000){\line(1,0){4.5}}
			 \linethickness{0.1mm}
			 \put(20.500000, -4.000000){\line(1,0){4}}
			 \linethickness{0.1mm}
			 \put(25.500000, -4.000000){\line(1,0){4}}
			 \linethickness{0.1mm}
			 \put(30.500000, -4.000000){\line(1,0){4}}
			                
		\end{picture}%
		\caption{Divisors $D_1$ and $D'_1$}
		\label{fig2}
	\end{figure}
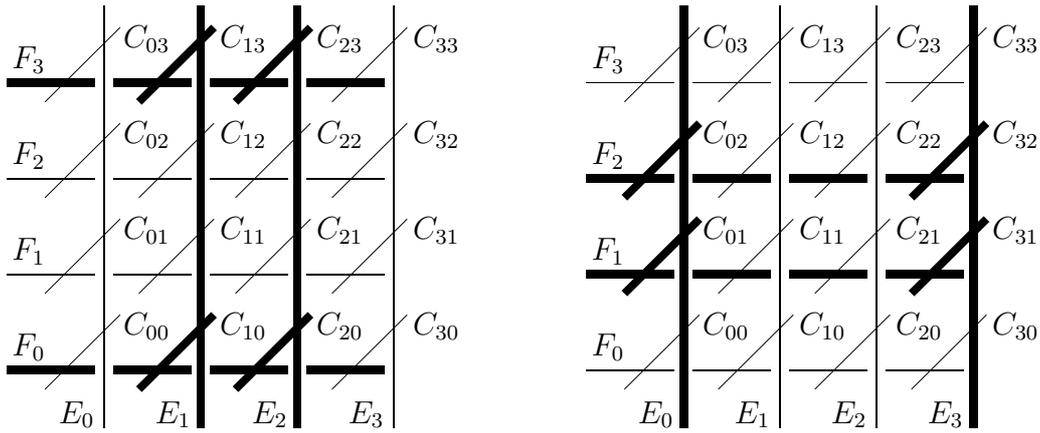

	We choose $C_{21}$ as the zero section of
	$\Phi_{|D_1|}$, turning it
    into  a Jacobian fibration.
	Let $F_{1, \eta}$ be the generic fiber of $\Phi_{|D_1|}$. 
	Then $(F_{1, \eta}, F_{1, \eta} \cap C_{21})$ is an elliptic curve with the origin $F_{1, \eta} \cap C_{21}$ over the function field $\bk(B_1)$. 
	The group of translations of the elliptic curve 
	$(F_{1, \eta}, F_{1, \eta} \cap C_{21})$ over $\bk(B_1)$ is called the Mordell-Weil group of $\Phi_{|D_1|}$, denoted by 
	${\rm MW} (\Phi_{|D_1|})$. The group ${\rm MW} (\Phi_{|D_1|})$ is an abelian group and it corresponds bijectively 
	to the set of sections of $\Phi_{|D_1|}$ in a natural way. 
	Moreover, as $X$ is a minimal surface,
	$${\rm MW} (\Phi_{|D_1|}) \subset {\rm Bir} (X/B_1) = {\rm Aut}(X/B_1) \subset {\rm Aut}(X).$$ 

Let
$$\tau : X \to X$$
be the involution induced by the involution
$$(x,x') \mapsto (x + \tau_3,x' + \tau'_3)$$
on $E \times F$.
We have
$$\tau(E_i) = E_{t(i)},  \ \ \tau(F_i) = F_{t(i)}, \ \ 
\text{hence }\tau(C_{ij}) = C_{t(i)t(j)},$$
where $t:\Set{0,1,2,3} \cto$ is the involution defined by
$$t(0) = 3, \ \ t(1) = 2, \ \ t(2) = 1, \ \ t(3) = 0.$$
Let us also notice that, as $\tau$ is a symplectic involution of the K3 surface $X$ (that is, $\tau^{\ast}|_{H^0(X, \Omega^2_X)} = \id$), the fixed point set of $\tau$ is made of exactly eight points (see e.g., \cite[Corollary 15.1.5]{Hu16}).
In particular, 
$\tau$ satisfies the assumptions of the following lemma.

\begin{lemma}\label{lem-translation}

Let $f$ be an automorphism of $X$ that preserves the fibration $\phi_{|D_1|}$, which descends to an automorphism of $B_1$ through $\phi_{|D_1|}$. Assume that $f(D_1)=D_1$, $f(D_1')=D_1'$, and $f$ acts freely on these two divisors. 
Assume moreover that the fixed locus of $f$ is finite and non-empty.
Then $f\in \mathrm{MW}(\phi_{|D_1|})$.
\end{lemma}

\begin{proof}
By assumption, there is an automorphism $g\in \Aut(B_1) \simeq \mathrm{PGL}(2,{\bf k})$ such that $\phi_{[D_1|}\circ f = g\circ\phi_{|D_1|}$. Since $f$ acts freely on $D_1$ and $D_1'$, and since it admits one fixed point $p \in X$, we have $\phi_{|D_1|}(D_1),\phi_{|D_1|}(D'_1)\ne \phi_{|D_1|}(p)$. So $g$ fixes three distinct points, hence $g=\id_{\mathbb{P}^1}$. So $f\in\mathrm{Aut}(X/B_1)$.

Finally, since the fixed locus of $f|_{F_{1,\eta}}$ is discrete, the linear part of $f|_{F_{1,\eta}}$ is trivial, i.e., $f|_{F_{1,\eta}}$ is a translation.
\end{proof}

\begin{lemma}\label{lem-tau1}
The involution $\tau$ coincides with the translation by
	the section $C_{12} \in \MW(\Phi_{|D_1|})$.
	In particular, $C_{12}$ is $2$-torsion in $\MW(\Phi_{|D_1|})$
	and
	$$C_{12} + C_{03} = C_{30}$$
	in $\MW(\Phi_{|D_1|})$.
\end{lemma}

\begin{proof} 

By Lemma \ref{lem-translation}, $\tau\in \MW(\Phi_{|D_1|})$. The remaining claims follow from $\tau(C_{21}) = C_{12}$ and $\tau(C_{03}) = C_{30}$.
\end{proof}

Let 
\begin{equation}\label{def:f}
f : X \to X
\end{equation}
be the translation by $C_{03}$. 
Since $f(C_{21})=C_{03}$ and since $f$ preserves $D_1$, we have $f(E_2)=F_3$. Hence, as a cyclic permutation of the 8-cycle made of the components of $D_1$, $f$ has order $4$. 
So $f^4$ stabilizes each component of $D_1$. 
For the affine coordinates on $E_2$ introduced in (\ref{eq33}), namely the one defined by 
\begin{equation}\label{eqn-xE}
x(E_2 \cap C_{20}) = 0, \ \ x(E_2 \cap C_{21}) = 1, \ \ x(E_2 \cap C_{22}) = s, \ \ x(E_2 \cap C_{23}) = \infty,    
\end{equation}
we have $f^4|_{E_2}(0)=0$ and $f^4|_{E_2}(\infty)=\infty$, so
\begin{equation}\label{eq42}
    f^4|_{E_2}(x) = r \cdot  x
\end{equation}
for some $r(s,t) \cnec r \in \bk^\times$.
This construction can be performed in family over the space of the parameters $(s,t)$, namely $(\bA^1_{\bk}\setminus\{0,1\})^2$. This yields that the scalar $r_{\bk}(s,t)$ is a rational function of $s,t$ defined over $\bk$. As this construction is compatible with extensions of the base field, it holds 
$r_{\C}|_{(\ol{\Q}\setminus\{0,1\})^2}=r_{\overline{\Q}}$, i.e., $r_{\C}$ is a rational function with coefficients in $\overline{\Q}$.

Before we continue, 
let us mentions the following lemma, 
which will be used several times.

\begin{lemma}\label{lem-symptrans}
Let $S$ be a K3 surface admitting an 
elliptic fibration $\Phi : S \to B$.
Let $\phi \in \Aut(X/B)$. If $\phi$ is symplectic, 
then $\phi$ is a translation by some element in $\MW(\Phi)$.
In particular, if $\phi \in \Aut(X/B)$ 
is a symplectic automorphism which fixes pointwisely a curve dominating $B$, 
then $\phi = \id_S$.
\end{lemma}

\begin{proof}

If $\phi$ has no fixed point $p$ in a general
fiber $F$ of $\Phi$, then
$\phi$ is already a translation.

Suppose that $\phi$ has a fixed point $p$ in a general (smooth) 
fiber $F$ of $\Phi$.
Since $\phi_*$ preserves the short exact sequence
$$0 \to T_{F,p} \to T_{S,p} \to (\Phi^*T_B)_{p} \to 0$$
and $\phi_*$ acts trivially on $(\Phi^*T_B)_{p}$, the assumption that $\phi$ is symplectic implies that it also acts trivially on $T_{F,p}$.
As $F$ is an elliptic curve, the linear part of $\phi|_{F}$ is the identity, and $\phi|_F$ fixes the point $p$, so $\phi|_{F}=\id_F$. 
Thus $\phi = \id_S$, in particular $\phi$ is a translation. 

The second statement follows immediately from the first one.
\end{proof}

\ssec{Computing $r_{\bk}(s,s)$}

Assume that $E = F$
(and therefore identify $\tau'_k$ with $\tau_k$, but keep denoting the vertical $(-2)$-curves in the configuration of Figure \ref{fig1} by $E_i$ and the horizontal ones by $F_j$).
Let $\gs : X \to X$
be the automorphism on $X$ induced by the automorphism
$$(x,x') \mapsto (x' + \tau_2,x + \tau_1)$$
on $E \times E$. 
Since this morphism has no fixed point of $E \times E$, the induced automorphism $\gs : X \to X$ has no fixed point neither.
 Moreover, we have
 $$\gs(E_i) = F_{s(i)},  \ \ \gs(F_i) = E_{s'(i)}, \ \ 
 \text{hence }\gs(C_{ij}) = C_{s'(j)s(i)},$$
 where $s:\Set{0,1,2,3} \cto$ is defined by
 $$s(0) = 2, \ \ s(1) = 3, \ \ s(2) = 0, \ \ s(3) = 1.$$
 and $s':\Set{0,1,2,3} \cto$ is defined by
 $$s'(0) = 1, \ \ s'(1) = 0, \ \ s'(2) = 3, \ \ s'(3) = 2.$$
So $\gs(D_1)=D_1$. In particular, $\gs$ preserves the fibration $\phi_{|D_1|}$.  Let us show that its induced action on the base $B_1\simeq\mathbb{P}^1$ is non-trivial. Assume by contradiction that it is trivial. Then, since $\mathrm{Fix}(\gs)$ is empty, the action of $\gs$ on the generic fiber must be a translation, i.e., $\gs\in \mathrm{MW}(\phi_{|D_1|})$. But then $\gs$ must be symplectic, contradiction! So $\gs$ acts non-trivially on $B_1$.

Note that the action of $\gs$ in the group $\Z_8$ of permutations of the components of $D_1$ is the same as that of $h^{-1}$, where $h : X \to X$ is the translation by $C_{33}$ with respect to $C_{21}$. Moreover, note that $h\circ\gs$ fixes the point $F_3\cap C_{33}$.

\begin{lemma}\label{lem-tausigma}
We have $h_{|D_1}^4 = \id_{D_1}$, and $h\circ\gs$ fixes $F_3$ pointwisely.
\end{lemma}

\begin{proof} 

First we note that $(h \circ \gs)^2_{|D_1}$ is trivial.
Indeed, note that $(h \circ \gs)^2$ is symplectic, and that it preserves the base
(because $h$ preserves the base and $\gs$ acts as an involution on the base).
Hence, by Lemma~\ref{lem-symptrans} $(h \circ \gs)^2$ is a translation.
In particular, by \cite[Theorem 9.1]{Ko63}, \cite[Paragraph 11.2.5, Corollary 11.2.4(ii)]{Hu16}, the restriction $(h \circ \gs)^2_{|(D_{1,\sm})}$ acts as an element of $\G_{m, \bk} \times {\mathbb Z}/8{\mathbb Z}$ on $D_{1,\sm}$.
Moreover, by construction, $h \circ \gs$
preserves $F_3\simeq\mathbb{P}^1$ and fixes three points of it (namely $F_3\cap C_{13}$, $F_3\cap C_{23}$, $F_3\cap C_{33}$), so $F_3\subset\mathrm{Fix}(h\circ \gs)$. 
Hence $(h \circ \gs)^2_{|D_{1,\sm}}$ is trivial.

Also, $h \circ \gs$ fixes the singular locus of $D_1$.
Assume by contradiction that a singular point $p \in D_1$ is an isolated fixed point. Then, as $h \circ \gs$ is an involution on $D_1$, the tangent map of $h \circ \gs$ at $p$ is $-\id$,
which contradicts the fact that $h \circ \gs$ is antisymplectic.
Moreover, $h \circ \gs$ cannot fix pointwisely both components of $D_1$ containing $p$. Therefore
$$F_0, E_1, E_2, F_3 \subset \Fix(h \circ \gs).$$
It follows that
$$\tau(E_2\cap C_{21})
= E_1 \cap C_{12}
= h \circ \gs (E_1 \cap C_{12})
= h(F_3 \cap C_{33})
= h^2(E_2\cap C_{21}).$$
Hence $(\tau^{-1}\circ h^2)_{|D_{1,\sm}}$ is a translation of the $\G_{m, \bk} \times {\mathbb Z}/8{\mathbb Z}$-torsor $D_{1,\sm}$ and it fixes a point, i.e., it is trivial.
As $\tau_{|D_1}$ is $2$-torsion, $h_{|D_1}$ is thus $4$-torsion.
\end{proof}

\begin{lemma}\label{lem-id4}
We have $r_{\bk}(s,s) = s^4$.
\end{lemma}

\begin{proof}

All the translations considered in this proof are restricted to $D_{1,\sm}$. Since we proved earlier that $h\circ\gs$ fixes $F_3$ pointwise, we have $h(E_2\cap C_{22})=F_3\cap C_{03}$. However, under the $\G_{m,\bk}$-action on $D_{1,\sm}$, the translation by $E_2\cap C_{22}$ correponds to the multiplication by $s$. Hence, we have 
$$f\circ h^{-1}(F_3 \cap C_{33}) = F_3 \cap C_{03} = s \cdot (F_3 \cap C_{33}).$$
and thus $f\circ h^{-1}(z)=s\cdot z$ for all $z\in D_{1,\sm}$.
Since the translation $h$ is 4-torsion, and since any two translations commute, we obtain
$$f^4(z) = (f\circ h^{-1})^4(z) = s^4 \cdot z$$
for all $z \in D_{1,\sm}$. Hence $r = s^4$.
\end{proof}

\ssec{The transcendence of $r$}

Now assume that $\bk = \C$.

\begin{proposition}\label{pro-nocst}
The map $(s,t) \mapsto r_\C(s,t) \in \C^\times$ is not constant.
As a consequence, as long as $s,t \in \C$ are algebraically independent,
$r_\C(s,t)$ is transcendental.
\end{proposition}

\begin{proof}
The first statement follows from 
Lemma~\ref{lem-id4}.
Since $r_\C$ is defined over $\ol{\Q}$,
$r_\C(s,t)$ is a non-constant rational function in $s$ and $t$ with coefficients
in $\ol{\Q}$. 
So $r_\C(s,t) \in \ol{\Q}$
implies that $s$ and $t$ are algebraically dependent.
\end{proof}

\ssec{Second Jacobian fibration}\label{subsec:second_jac}

We assume $\bk = \C$ for simplicity. 

    Let $D_2$ be the divisor on $X$ defined by 
	$$D_2 := F_0 + C_{30} + E_3 + C_{31} + F_1 + C_{21} + E_2 + C_{20};$$
	see Figure \ref{fig3}.
	By the same argument as in Subsection~\ref{subsec:first_jac}, $|D_2|$ defines an elliptic fibration
	$${\Phi}_{|D_2|} : X \to B_2 := {\mathbb P}^1,$$
	containing both $D_2$ and 
	$$D'_2 := E_0 + C_{03} + F_3 + C_{13} + E_1 + C_{12} + F_2 + C_{02}$$
	as fibers.
	Note that ${\Phi}_{|D_2|}$ has sections $C_{23}$ and $C_{32}$.
	
	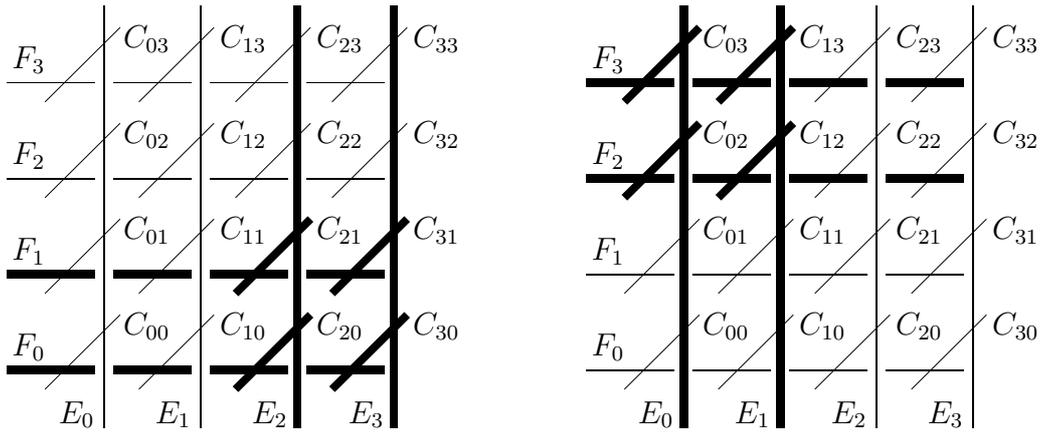
\begin{figure}
		\unitlength 0.1in
		\begin{picture}(25.000000,24.000000)(-1.000000,-23.500000)
		
			\put(-10.500000, -22.000000){\makebox(0,0)[rb]{$E_0$}}%
			\put(-5.500000, -22.000000){\makebox(0,0)[rb]{$E_1$}}%
			\put(-0.500000, -22.000000){\makebox(0,0)[rb]{$E_2$}}%
			\put(4.500000, -22.000000){\makebox(0,0)[rb]{$E_3$}}%
			\put(-14.750000, -18.500000){\makebox(0,0)[lb]{$F_0$}}%
			\put(-14.750000, -13.500000){\makebox(0,0)[lb]{$F_1$}}%
			\put(-14.750000, -8.500000){\makebox(0,0)[lb]{$F_2$}}%
			\put(-14.750000, -3.500000){\makebox(0,0)[lb]{$F_3$}}%
			\put(-9.000000, -16.000000){\makebox(0,0)[lt]{$C_{00}$}}%
			\put(-9.000000, -11.000000){\makebox(0,0)[lt]{$C_{01}$}}%
			\put(-9.000000, -6.000000){\makebox(0,0)[lt]{$C_{02}$}}%
			\put(-9.000000, -1.000000){\makebox(0,0)[lt]{$C_{03}$}}%
			\put(-4.000000, -16.000000){\makebox(0,0)[lt]{$C_{10}$}}%
			\put(-4.000000, -11.000000){\makebox(0,0)[lt]{$C_{11}$}}%
			\put(-4.000000, -6.000000){\makebox(0,0)[lt]{$C_{12}$}}%
			\put(-4.000000, -1.000000){\makebox(0,0)[lt]{$C_{13}$}}%
			\put(1.000000, -16.000000){\makebox(0,0)[lt]{$C_{20}$}}%
			\put(1.000000, -11.000000){\makebox(0,0)[lt]{$C_{21}$}}%
			\put(1.000000, -6.000000){\makebox(0,0)[lt]{$C_{22}$}}%
			\put(1.000000, -1.000000){\makebox(0,0)[lt]{$C_{23}$}}%
			\put(6.000000, -16.000000){\makebox(0,0)[lt]{$C_{30}$}}%
			\put(6.000000, -11.000000){\makebox(0,0)[lt]{$C_{31}$}}%
			\put(6.000000, -6.000000){\makebox(0,0)[lt]{$C_{32}$}}%
			\put(6.000000, -1.000000){\makebox(0,0)[lt]{$C_{33}$}}%
			
			\linethickness{0.1mm}
			\put(-10.000000, -22.000000){\line(0,1){22}}
                          \linethickness{0.1mm}
			 \put(-5.000000, -22.000000){\line(0,1){22}}
                          \linethickness{1mm}
			 \put(0.000000, -22.000000){\line(0,1){22}}
			 \linethickness{1mm}
			 \put(5.000000, -22.000000){\line(0,1){22}}
                          \linethickness{1mm}
			 \put(-15.000000, -19.000000){\line(1,0){4.5}}
			\linethickness{1mm}
			\put(-9.500000, -19.000000){\line(1,0){4}}
                         \linethickness{1mm}
			\put(-4.500000, -19.000000){\line(1,0){4}}
                         \linethickness{1mm}
			\put(0.500000, -19.000000){\line(1,0){4}}
			 \linethickness{1mm}
			 \put(-15.000000, -14.000000){\line(1,0){4.5}}
			 \linethickness{1mm}
			 \put(-9.500000, -14.000000){\line(1,0){4}}
			 \linethickness{1mm}
			 \put(-4.500000, -14.000000){\line(1,0){4}}
			\linethickness{1mm}
			 \put(0.500000, -14.000000){\line(1,0){4}}
			 \linethickness{0.1mm}
			 \put(-15.000000, -9.000000){\line(1,0){4.5}}
			 \linethickness{0.1mm}
			 \put(-9.500000, -9.000000){\line(1,0){4}}
			 \linethickness{0.1mm}
			 \put(-4.500000, -9.000000){\line(1,0){4}}
			 \linethickness{0.1mm}
			 \put(0.500000, -9.000000){\line(1,0){4}}
			 \linethickness{0.1mm}
			 \put(-15.000000, -4.000000){\line(1,0){4.5}}
			 \linethickness{0.1mm}
			 \put(-9.500000, -4.000000){\line(1,0){4}}
			 \linethickness{0.1mm}
			 \put(-4.500000, -4.000000){\line(1,0){4}}
			 \linethickness{0.1mm}
			 \put(0.500000, -4.000000){\line(1,0){4}}
			                
                \put(19.500000, -22.000000){\makebox(0,0)[rb]{$E_0$}}%
			\put(24.500000, -22.000000){\makebox(0,0)[rb]{$E_1$}}%
			\put(29.500000, -22.000000){\makebox(0,0)[rb]{$E_2$}}%
			\put(34.500000, -22.000000){\makebox(0,0)[rb]{$E_3$}}%
			\put(15.250000, -18.500000){\makebox(0,0)[lb]{$F_0$}}%
			\put(15.250000, -13.500000){\makebox(0,0)[lb]{$F_1$}}%
			\put(15.250000, -8.500000){\makebox(0,0)[lb]{$F_2$}}%
			\put(15.250000, -3.500000){\makebox(0,0)[lb]{$F_3$}}%
			\put(21.000000, -16.000000){\makebox(0,0)[lt]{$C_{00}$}}%
			\put(21.000000, -11.000000){\makebox(0,0)[lt]{$C_{01}$}}%
			\put(21.000000, -6.000000){\makebox(0,0)[lt]{$C_{02}$}}%
			\put(21.000000, -1.000000){\makebox(0,0)[lt]{$C_{03}$}}%
			\put(26.000000, -16.000000){\makebox(0,0)[lt]{$C_{10}$}}%
			\put(26.000000, -11.000000){\makebox(0,0)[lt]{$C_{11}$}}%
			\put(26.000000, -6.000000){\makebox(0,0)[lt]{$C_{12}$}}%
			\put(26.000000, -1.000000){\makebox(0,0)[lt]{$C_{13}$}}%
			\put(31.000000, -16.000000){\makebox(0,0)[lt]{$C_{20}$}}%
			\put(31.000000, -11.000000){\makebox(0,0)[lt]{$C_{21}$}}%
			\put(31.000000, -6.000000){\makebox(0,0)[lt]{$C_{22}$}}%
			\put(31.000000, -1.000000){\makebox(0,0)[lt]{$C_{23}$}}%
			\put(36.000000, -16.000000){\makebox(0,0)[lt]{$C_{30}$}}%
			\put(36.000000, -11.000000){\makebox(0,0)[lt]{$C_{31}$}}%
			\put(36.000000, -6.000000){\makebox(0,0)[lt]{$C_{32}$}}%
			\put(36.000000, -1.000000){\makebox(0,0)[lt]{$C_{33}$}}%
			
			\linethickness{1mm}
			\put(20.000000, -22.000000){\line(0,1){22}}
                          \linethickness{1mm}
			 \put(25.000000, -22.000000){\line(0,1){22}}
                          \linethickness{0.1mm}
			 \put(30.000000, -22.000000){\line(0,1){22}}
			 \linethickness{0.1mm}
			 \put(35.000000, -22.000000){\line(0,1){22}}
                          \linethickness{0.1mm}
			 \put(15.000000, -19.000000){\line(1,0){4.5}}
			\linethickness{0.1mm}
			\put(20.500000, -19.000000){\line(1,0){4}}
                         \linethickness{0.1mm}
			\put(25.500000, -19.000000){\line(1,0){4}}
                         \linethickness{0.1mm}
			\put(30.500000, -19.000000){\line(1,0){4}}
			 \linethickness{0.1mm}
			 \put(15.000000, -14.000000){\line(1,0){4.5}}
			 \linethickness{0.1mm}
			 \put(20.500000, -14.000000){\line(1,0){4}}
			 \linethickness{0.1mm}
			 \put(25.500000, -14.000000){\line(1,0){4}}
			\linethickness{0.1mm}
			 \put(30.500000, -14.000000){\line(1,0){4}}
			 \linethickness{1mm}
			 \put(15.000000, -9.000000){\line(1,0){4.5}}
			 \linethickness{1mm}
			 \put(20.500000, -9.000000){\line(1,0){4}}
			 \linethickness{1mm}
			 \put(25.500000, -9.000000){\line(1,0){4}}
			 \linethickness{1mm}
			 \put(30.500000, -9.000000){\line(1,0){4}}
			 \linethickness{1mm}
			 \put(15.000000, -4.000000){\line(1,0){4.5}}
			 \linethickness{1mm}
			 \put(20.500000, -4.000000){\line(1,0){4}}
			 \linethickness{1mm}
			 \put(25.500000, -4.000000){\line(1,0){4}}
			 \linethickness{1mm}
			 \put(30.500000, -4.000000){\line(1,0){4}}
			                
		\end{picture}%
		\caption{Divisors $D_2$ and $D'_2$}
		\label{fig3}
	\end{figure}

    We choose $C_{23}$ as the zero section of
	$\Phi_{|D_2|}$, turning it
    into  a Jacobian fibration.	
	Let $F_{2, \eta}$ be the generic fiber of $\Phi_{|D_2|}$. Then $(F_{2, \eta}, F_{2,\eta} \cap C_{23})$ is an elliptic curve with the origin $F_{2, \eta} \cap C_{23}$ over $\bk(B_2)$.

\begin{lemma}\label{lem-tau}
	The section $C_{32}$ is a $2$-torsion element in $\MW(\Phi_{|D_2|})$.
\end{lemma}

\begin{proof} 

The proof is essentially the same as the argument of Lemma ~\ref{lem-tau1}.
Instead of
$\tau$, we consider the symplectic involution $\nu$
induced by
$$(x,x') \mapsto (x + \tau_1,x' + \tau'_1)$$
on $E \times F$, so that
$$\nu(E_i) = E_{t'(i)},  \ \ \nu(F_i) = F_{t'(i)}, \ \ 
\nu(C_{ij}) = C_{t'(i)t'(j)},$$
where $t':\Set{0,1,2,3} \cto$ is the involution defined by
$$t'(0) = 1, \ \ t'(1) = 0, \ \ t'(2) = 3, \ \ t'(3) = 2.$$
We also replace $D_1$ by $D_2$, and $D_1'$ by $D_2'$ 
in the proof.
\end{proof}

\begin{proof}[Another proof that $C_{32}$ is torsion.]

Later in this paper, we only need the weaker statement that
$C_{32}$ is torsion.
Here we provide another proof of it,
using Shioda's height pairing.

		We compute the Shioda's height paring value $\langle C_{32}, C_{32} \rangle$ of the section $C_{32}$ of $\Phi_{|D_2|}$ with respect to the zero section $C_{23}$ by using the formula \cite[Theorem 8.6]{Sh90}. To use this, first note that the reducible fibers of $\Phi_{|D_2|}$ are $D_2$ and $D_2'$,
		as $\Phi_{|D_2|}$ is of Type ${\mathcal I}_1$ in \cite[Table 2 in Page 662]{Og89}.
		The zero section $C_{23}$ meets $D_2$ and $D_2'$ at $E_2$ and $F_3$ respectively, while the section $C_{32}$ meets $D_2$ and $D_2'$ at $E_3$ and $F_2$ respectively. 
		Thus, by \cite[Theorem 8.6, Table 8.16]{Sh90}, we compute that
		$$\langle C_{32}, C_{32} \rangle = 2\cdot2 + 2 \cdot 0 - 2 \cdot \frac{4(8-4)}{8} = 0.$$
		Thus $C_{32}$ corresponds to a torsion element of ${\rm MW} (\Phi_{|D_2|})$ by \cite[Equation 8.10]{Sh90}. 
\end{proof}

	Consider the inversion $\psi$ of the elliptic curve $(F_{2, \eta}, F_{2,\eta} \cap C_{23})$. Then 
	\begin{equation}\label{def:psi}
	\psi \in {\rm Bir} (X/B_2) = {\rm Aut}(X/B_2) 
	\subset {\rm Aut}(X).
	\end{equation}
	As $\psi$ fixes the section $C_{23}$ pointwisely and is not trivial,
	it is not symplectic
	by Lemma~\ref{lem-symptrans}.
	Since $\psi^2 = \id_X$, we have
	$\psi^*\omega_X = -\omega_X$,
	where $\go_X$ is a holomorphic symplectic form on $X$.

    Set
	$$\psi_n := f^{-4n} \circ \psi \circ f^{4n} \in {\rm Aut}(X).$$
	
    	\begin{lemma}
	$\psi_n$ are involutions and $\psi_n \in {\rm Ine} (X,C_{23})$.
	\end{lemma}

\begin{proof}
    Since $\psi$ is of order two by definition, so are $\psi_n$. 	
	Recall that each component of $D_1$ (in particular, $C_{23}$) is stable under $f^{4}$, 
	we have
	$f^{4n}\in {\rm Dec} (X, C_{23})$ for all $n$. 
	Combining this with $\psi \in {\rm Ine}(X, C_{23})$ 
	by the definition of $\psi$, we deduce that 
		$$\psi_n = f^{-4n} \circ \psi \circ f^{4n} \in 
		{\rm Ine}(X, C_{23}).$$
	This completes the proof. 
\end{proof}
	
	\begin{proposition}\label{prop42} 
		 Assume that the parameters $s,t$ are such that
		 $r$ is not a root of unity (e.g., $s,t$ are algebraically independent). Then
			$\psi_n \not= \psi_m$ whenever $n \not= m$. 
	\end{proposition}
	
	\begin{proof} 
		
		Since $\psi(D_2) = D_2$ and $\psi|_{C_{23}}$ is the identity, and since $E_2$ is the only irreducible component of $D_2$ containing $D_2 \cap C_{23}$,
		we have $\psi(E_2) = E_2$. As $f^4(E_2) = E_2$, 
		$$\psi_n = f^{-4n} \circ \psi \circ f^{4n}\in {\rm Dec} (X, E_2).$$ 
		As $\psi$ is an antisymplectic involution which fixes $C_{23}$ pointwise, its linearization at the point $E_2\cap C_{23}$ has eigenvalues $1$ (in the direction corresponding to $C_{23}$) and $-1$; in particular, it cannot fix $E_2$ pointwise. Therefore, $\psi$ fixes at most two points on $E_2$, and as $\psi(D_2) = D_2$,
		this implies
		$$\psi(E_2 \cap C_{21}) = E_2 \cap C_{20}, \ \ \psi(E_2 \cap C_{20}) = E_2 \cap C_{21}, \ \ \psi(E_2 \cap C_{23}) = E_2 \cap C_{23}.$$
		In terms of the affine coordinate $x$ of $E_2$, defined by~\eqref{eqn-xE}, we have 
		$$\psi(1) = 0, \ \ \psi(0) = 1, \ \ \psi(\infty) = \infty,$$
		so $\psi(x) = 1- x$, and thus 
		together with (\ref{eq42}), we conclude that
		\begin{equation}\label{eqn-psicoord}
		    \psi_n(x) = \frac{1}{r^n} - x.
		\end{equation}
		As $r$ is not a root of unity,
		this proves the assertion.  
	\end{proof}

	The next proposition will be useful when we prove Theorem~\ref{thm32}.

	\begin{proposition}\label{prop43} 

	Assume that the parameters $s,t$ are such that $E$ is not isogenous to $F$, and $r$ is not a root of unity (this holds for very general $s,t$). Then
		\begin{enumerate}
			\item There is a nef and big curve $\Sigma \subset X$ such that $\gamma(\Sigma) = \Sigma$ for all 
				$$\gamma \in {\rm Cent}(\psi) := \{g \in {\rm Aut}(X)\,|\, g \circ \psi = \psi \circ g\}.$$
			In particular, ${\rm Cent}(\psi)$ is a finite group. 
			\item The set of conjugacy classes of 
			$${\mathcal S} := \{\psi_n\,|\, n \in {\mathbb Z}\} \subset {\rm Aut}(X)$$
			in ${\rm Ine} (X, C_{23})$ is an infinite set. 
		\end{enumerate}
	\end{proposition}
	
	\begin{proof}
		
		Let us show (1). We will show that $X^{\psi}$ contains a unique irreducible smooth curve of genus $g \ge 2$, which we will denote by $\Sigma$. Note that then $\Sigma$ is nef and big, as
		$$\Sigma^2 = 2g(\Sigma) - 2 >0.$$
		So, provided the existence and uniqueness of $\Sigma$, it clearly follows that $\gamma(\Sigma) = \Sigma$ for all $\gamma \in {\rm Cent}(\psi)$ and, by the fact that $\Sigma$ is nef and big, it also follows that ${\rm Cent}(\psi)$ is finite by \cite[Proposition 2.25]{Br18}.
		
		Let us now establish the existence and uniqueness of $\Sigma$. 
		Since $\psi$ is an involution which
		satisfies $\psi^* \omega_X = -\omega_X$, 
		by~\cite[Lemma 1]{Be11}
		the fixed point locus 
		$X^{\psi}$ is either empty, or a disjoint union of smooth irreducible curves.
		As $C^2 > 0$ for any smooth irreducible curve $C \subset X$
		of genus greater or equal to $2$, it follows from the Hodge index theorem that any two curves of genus higher or equal to 2 in $X$ intersect.
		In particular, $X^{\psi}$ has at most one component $\gS$ of genus $g \ge 2$.
		
		It remains to show that such a component exists.
		By~\cite[Theorem 2.1]{Og89},
		since $E$ and $F$ are not isogenous and since $\Phi_{|D_2|}$ is of Type ${\mathcal J}_1$, we have 
		$${\rm MW} (\Phi_{|D_2|}) = {\mathbb Z}^2 \oplus {\mathbb Z}/2{\mathbb Z}.$$
		By Lemma~\ref{lem-tau},
		the unique non-trivial torsion element in ${\rm MW} (\Phi_{|D_2|})$ is thus $C_{32}$. 
		Note that $X^{\psi}$ intersects every smooth fiber of $\Phi_{|D_2|}$ at its four torsion points. Hence, the irreducible components of $X^{\psi}$ which dominate $B_2$ are the zero section $C_{23}$, the unique $2$-torsion section $C_{32}$, 
		and the closure $\Sigma$ in $X$ of the set of remaining $2$-torsion points
		in the fibers of $\Phi_{|D_2|}$; necessarily $\gS$ is irreducible. 
		
		Let us study the ramification of the double cover $\Phi_{|D_2|}|_{\Sigma}:\Sigma\to\mathbb{P}^1$. For that, we need to understand the singular fibers of $\Phi_{|D_2|}$. Let us prove that there are no fibers of type II for $\Phi_{|D_2|}$: Then by \cite[Theorem 2.1]{Og89}, the singular fibers are two fibers of type $\mathrm{I}_8$ and eight fibers of type $\mathrm{I}_1$. Note that, by \cite[Corollary 11.2.4]{Hu16}, the smooth part of a type II singular fiber is isomorphic to $(\C,+)$, the group action of it being compatible with the action of the Mordell-Weil group. As $(\C,+)$ has no torsion element, the non-trivial 2-torsion section in the Mordell-Weil group must act trivially in this singular fiber, hence it yields a symplectic automorphism (namely a translation) of order two fixing a whole curve pointwise, contradiction. This implies that $\Phi_{|D_2|}$ has no type II singular fibers. 
		Hence, by~\cite[Theorem 2.1]{Og89}, $\Phi_{|D_2|}$ has eight singular fibers of type $\mathrm{I}_1$ (i.e., rational curves with a node). For each of these singular fibers, the smooth locus is isomorphic to $(\C^*,\times)$, with multiplication by $-1$ (fixing the node) corresponding to the translation by the unique non-trivial 2-torsion element $C_{32}$, and with the inversion $\psi$ corresponding to the map $z\mapsto\frac{1}{z}$ (fixing the node). In particular, $\psi$ fixes the intersection of the singular fiber with $C_{23}$, with $C_{32}$, and the node of the singular fiber. But as $\psi$ is antisymplectic, its fixed locus is a disjoint union of smooth curves, in particular the nodes of the singular fibers belong to horizontal components of $X^{\psi}$. But the horizontal components of $X^{\psi}$ are $C_{23}$, $C_{32}$ and $\Sigma$. So the eight nodes all belong to $\Sigma$. Thus the projection $\Sigma \to \mathbb{P}^1$ is a double cover branched 
		at at least eight points,
		so $\Sigma$ has genus at least $3$ by Riemann-Hurwitz's formula. 
		This completes the proof of (1).
		
		Now we show (2). Recall that $\psi_n \not= \psi_m$ if $n \not= m$ by Proposition \ref{prop42}. So, as in \cite[Lemma 4.5]{DO19}, it suffices to show that for each fixed $n$, there are only finitely many $m$ such that 
		\begin{equation}\label{eq45}
			\psi_m = h^{-1} \circ \psi_n \circ h
		\end{equation}
		for some $h \in {\rm Ine} (X, C_{23})$. 
		
		Since $\psi_n = f^{-4n} \circ \psi \circ f^{4n}$, we have from (\ref{eq45}) that
		$$f^{-4m} \circ \psi \circ f^{4m} = h^{-1} \circ f^{-4n} \psi \circ f^{4n} \circ h$$
		and equivalently
		$$f^{4n} \circ h \circ f^{-4m} \circ \psi = \psi \circ f^{4n} \circ h \circ f^{-4m}.$$
		Thus 
		$$f^{4n} \circ h \circ f^{-4m} \in {\rm Cent} (\psi).$$
		
		Recall that $f^4$ fixes each component of $D_1$, and thus acts as the multiplication by $r\in\C^*$ on the smooth part of $D_1$ identified to the group $\C^*\times\Z/8\Z$.
		Then from $x(E_2 \cap C_{23}) = \infty$, $f^{\ast}\go_X = \go_X$ and $f^4|_{E_2}(x) = r \cdot  x$,
		there exists an affine coordinate $z$ on $C_{23}$,
		with $z(C_{23} \cap E_2) = 0$, and $z(C_{23} \cap F_3) = \infty$, such that
		$$f^{4n}|_{C_{23}}(z) = r^n \cdot z$$ 
		for all $n$. Since $h|_{C_{23}}(z) = z$ as $h \in {\rm Ine}(X, C_{23})$, it follows that 
		$$f^{4n} \circ h \circ f^{-4m}|_{C_{23}} (z) = r^{n-m} \cdot z.$$
		Since ${\rm Cent} (\psi)$ is a finite set, necessarily
		$${\mathcal R}_n := \Set{r^{n-m} | f^{4n} \circ h \circ f^{-4m} \in {\rm Cent} (\psi)}$$
		is finite as well. 
		As $r$ is not a root of unity, 
		it follows that for each fixed $n$, 
		there are only finitely many integers $m$ satisfying 
		$f^{4n} \circ h \circ f^{-4m} \in {\rm Cent} (\psi)$. 
		This completes the proof of (2).   
	\end{proof}
	
	\section{Mukai's Enriques surfaces over $\bk$}\label{sct4}

	Let $\bk$ be a field of characteristic zero.
	We continue the constructions and use the same notations 
	as in Subsection~\ref{ssec-Kummer}.
	Throughout Section~\ref{sct4}, we assume that 
	$$s \ne t,$$
	where
	$s,t$ are the parameters 
	in~\eqref{eq31} and~\eqref{eq32} defining $E$ and $F$.

	\subsection{Mukai's Enriques surfaces}\label{sct3sub2}

	In this subsection, following Mukai \cite{Mu10},
	we recall the construction of an Enriques surface $Z$ from 
	a certain Kummer surface of product type.
	We repeat the construction to emphasize that
	Mukai's construction works over any field $\bk$ of characteristics zero,
	and to
	fix some notations.
	
	Set
	$$\theta := [(1_E, -1_F)] = [(-1_E, 1_F)] \in {\rm Aut}(X).$$
	Then $\theta$ is an automorphism of $X$ of order $2$.
	Let $$T := X/\langle \theta \rangle, \ \text{ and } \ q : X \to T$$
	be the quotient surface and the quotient morphism. 
	Then $T$ is a smooth projective surface, 
	and each $q(C_{ij})$ ($0 \le i, j \le 3$) is a smooth $(-1)$-curve over $\bk$. 
	By construction, blowing down $T$ along these $16$ $(-1)$-curves
	is 
	$$ (E/\pm 1_E) \times (F/\pm 1_F) \simeq \P^1 \times \P^1,$$
	and each $C_{ij} \subset X$ gets contracted to 
	a $\bk$-point
	$$p_{ij} \in \P^1 \times \P^1$$ 
	under the composition
	$$X \to T \to \P^1 \times \P^1.$$

	Consider the Segre embedding
	$$Q \cnec \P^1 \times \P^1 \subset \P^3,$$
	and the two sets
	$$\{ P_{i0},P_{i1}, P_{i2}, P_{i3}\}\subset E_i\cong \P^1,\ \ \{P'_{0j}, P'_{1j}, P'_{2j}, P'_{3j}\} \subset F_j\cong \P^1.$$ 
	Since $s \ne t$
	by assumption, 
	the two ordered $4$-tuples 
	$$(P_{i0},P_{i1}, P_{i2}, P_{i3}),\ \ (P'_{0j}, P'_{1j}, P'_{2j}, P'_{3j})$$ 
	are not projectively equivalent.
	In other words,
	the four $\bk$-points 
	$p_{00}, p_{11}, p_{22}, p_{33} \in Q$  are not coplanar in $\P^3$. Moreover, none of the lines passing through two of these four points is included in the quadric $Q$.
	We may therefore choose the homogeneous coordinates 
	$[w_1 : w_2 : w_3 : w_4]$
	of 
	$$\P^3 = {\rm Proj}\, \bk[w_1, w_2, w_3, w_4]$$
	so that
	$$p_{00}=[1:0:0:0],\ \ p_{11}= [0:1:0:0],\ \ p_{22}=[0:0:1:0],\ \ p_{33}=[0:0:0:1].$$
	Then, up to multiplying 
	$w_i$ by some element in $\bk^\times$ if necessarily, the equation of $Q$ is written in the form
	\begin{equation}
	\alpha_1w_2w_3 + \alpha_2w_1w_3 + \alpha_3w_1w_2 + (w_1+w_2+w_3)w_4 = 0 
	\end{equation}
	for some $\alpha_i \in \bk^\times$ satisfying the smoothness (non-degeneration) condition
	$$\alpha_1^2 + \alpha_2^2 + \alpha_3^2 -2\alpha_1\alpha_2 - 2\alpha_1\alpha_3 - 2\alpha_2\alpha_3 \not= 0.$$
	Then the Cremona involution of $\P^3$
	$$\tilde{\tau}' : [w_1 :w_2 : w_3 : w_4] \mapsto [\alpha_1w_2w_3w_4 : \alpha_2w_1w_3w_4 : \alpha_3w_1w_2w_4 : \alpha_1\alpha_2\alpha_3w_1w_2w_3]$$
	is defined over $\bk$ and satisfies
	$\tilde{\tau}'(Q) = Q$.
	Hence we obtain a birational automorphism
	$$\tau' := \tilde{\tau}'|_{Q} \in {\rm Bir} (Q/\bk).$$
	By the definition of $\tau'$, one can readily check the following facts (\cite[Section 2]{Mu10}).
	\begin{lemma}\label{lem32}
	\hfill
	
		\begin{enumerate}
			\item The indeterminacy locus of $\tau'$ consists of the four points $p_{00},p_{11},p_{22},p_{33}$, and $\tau'$ contracts the conic  $C'_{i} := Q \cap 
			(w_i= 0)$ to $p_{ii}$ ($0 \le i \le 3$).
			\item $\tau'$ interchanges the two lines through $p_{ii}$ for each $i=0$, $1$, $2$, $3$.
			\item $\mu^{-1}\circ \tau' \circ \mu \in {\rm Aut}(B/\bk)$, where $\mu : B \to \P^1 \times \P^1$ is the blow-up at the four $\bk$-points $p_{ii}$ ($0 \le i \le 3$).
		\end{enumerate}
	\end{lemma}

	By Lemma \ref{lem32} (2), $\tau'(p_{ij}) = p_{ji}$ if $0 \le i \not= j \le 3$. Therefore $\tau'$ lifts to 
	$$\tilde{\tau} \in {\rm Aut}(T/\bk)$$
	by Lemma \ref{lem32} (3).
	Since $q : X \to T$ is the finite double cover branched along the unique anti-bicanonical divisor $$\sum_{i=0}^{3} (q(E_i)+q(F_i))\in |-2K_T|,$$
	it follows that $\tilde{\tau}$ lifts to an involution 
	\begin{equation}\label{def_eps}
		\epsilon \in {\rm Aut}(X/\bk).
	\end{equation}
	A priori, there are exactly two choices of the lifting $\epsilon$; 
	if we denote one lifting by $\epsilon_0$ then the other is $ \theta \circ \epsilon_0$. 
	Let $\go_X$ be a generator of $H^0(X,\gO_X^2)$.
	Since $\theta^*\omega_X = -\omega_X$ and $g^*\omega_X = \pm \omega_X$ 
	for any involution $g: X \cto$, we choose the unique lift $\epsilon$ with $\epsilon^* \omega_X = -\omega_X$.
	Let
	\begin{equation}\label{def_Z}
		Z := X/\langle \epsilon \rangle, \ \text{ and } \ \pi : X \to Z
	\end{equation}
	be the quotient surface and the quotient morphism.
	The following theorem, which is crucial for us, was proven by Mukai \cite[Proposition 2]{Mu10}.
	
	\begin{thm}\label{thm31}
		The involution $\epsilon$ acts freely on $X$ and $Z$ is an Enriques surface.  
	\end{thm}
	
	Note that the involution $\epsilon$ does not come from any involution of the Kummer quotient $E\times F/\langle -1_{E\times F}\rangle$, 
	since it does not preserve the set of exceptional divisors (and in particular the $C_{ii}$ for $0\le i\le 3$) of the birational map $X\to E\times F/\langle -1_{E\times F}\rangle$.

	\begin{lemma}\label{lem-eps} 
	
	The involution $\epsilon \in {\rm Aut}(X)$ satisfies

		\begin{enumerate} 
			\item $\epsilon (E_i) = F_i$, $\epsilon(F_i) = E_{i}$ for all $i =0$, $1$, $2$, $3$.
			\item $\epsilon(C_{ij}) = C_{ji}$ for all $0\leq i \not= j \leq 3$. 
		\end{enumerate}
	\end{lemma}
	\begin{proof} Both statements follow from the  
	constructions of $\tilde{\tau}$ and $\epsilon$. 
	\end{proof}

 \ssec{Descending $\psi_n$}
 
 We assume that $\bk$ is algebraically closed.
    
    \begin{lemma}\label{lem-descent-psin}
	Each $\psi_n$ descends to an automorphism on $Z$
	fixing $D_{23} \cnec \pi(C_{23})$ pointwise. 
	\end{lemma}

\begin{proof} 

Recall that $D_i$ and $D'_i$ are the two fibers of $\Phi_{|D_i|}$ of type $\mathrm{I}_8$ (stemming from our $(-2)$-curves configuration). By Lemma~\ref{lem-eps},
        we have $\epsilon(D_i) = D'_i$ for $i \in \Set{1,2}$, so $\epsilon$ preserves the fibration $\Phi_{|D_i|}$, while acting on the base $B_i$ of $\Phi_{|D_i|}$ as a non-trivial involution.
        It follows that 
		$$\tilde{f} :=  f^{-1} \circ \epsilon^{-1} \circ f \circ \epsilon \in {\rm Aut}(X/B_1),$$
		$$\tilde{\psi} := \epsilon^{-1} \circ \psi^{-1} \circ \epsilon \circ \psi \in {\rm Aut}(X/B_2).$$
		(See (\ref{def:f}) and (\ref{def:psi}) for the definitions of $f$ and $\psi$.)
		Since $f$ is symplectic, while $\psi$ and $\epsilon$ are antisymplectic, we have
		$$\tilde{f}^{*}\omega_X = \omega_X \ \ \text{and} \ \ \tilde{\psi}^{*}\omega_X = \omega_X.$$

		Recall that $C_{12}$ is a $2$-torsion element of ${\rm MW}(\Phi_{|D_1|})$ with respect to the zero section $C_{21}$ by Lemma~\ref{lem-tau1}, 
		and that $f$ is the translation by $C_{03}$. Thus, under 
		$\tilde{f} :=  f^{-1} \circ \epsilon^{-1} \circ f \circ \epsilon$,  
		we have: 
		$$C_{12} \mapsto C_{21} \mapsto C_{03} \mapsto C_{30} \mapsto C_{30} - C_{03} = C_{12},$$
		where the last equality follows from Lemma~\ref{lem-tau1}. 
		As $\tilde{f}$ is symplectic and fixes a section $C_{12}$,
		by Lemma~\ref{lem-symptrans}, we have 
		$\tilde{f} = \id_X$, namely 
		\begin{equation}\label{eq43}
			f \circ \epsilon = \epsilon \circ f.
		\end{equation}
		
		Similarly, as $C_{32}$ is a $2$-torsion element of ${\rm MW}(\Phi_{|D_2|})$ with respect the zero section $C_{23}$ by Lemma~\ref{lem-tau}, 
		and $\psi$ is the inversion with respect to the zero section $C_{23}$, 
		we have 
		$$\psi(C_{23}) = C_{23}, \ \ \psi(C_{32}) = C_{32}.$$ 
		Thus, under $\tilde{\psi} := \epsilon^{-1} \circ \psi^{-1} \circ \epsilon \circ \psi$, we have
		$$C_{23} \mapsto C_{23} \mapsto C_{32} \mapsto C_{32} \mapsto C_{23}.$$ 
		Again by Lemma~\ref{lem-symptrans}, we have $\tilde{\psi} = \id_X$, namely 
		\begin{equation}\label{eq44}
			\psi \circ \epsilon = \epsilon \circ \psi.
		\end{equation}
		By (\ref{eq43}) and (\ref{eq44}), $\psi_n = f^{-4n} \circ \psi \circ f^{4n}$ also commutes with $\epsilon$. Hence $\psi_n \in {\rm Ine}(X, C_{23})$ descends to an element of ${\rm Ine} (Z, D_{23})$. 
\end{proof}

	\section{Surfaces with infinitely many real forms}\label{sct5}
	
	We keep the same notations as in Sections~\ref{sec-Kummer} and~\ref{sct4}.
	In this section,
	we work over $\bk = \C$
	and make
	the following assumption on 
	the parameters
	$s,t$ in~\eqref{eq31} and~\eqref{eq32} defining $E$ and $F$.
	
	\begin{assumption}\label{ass31} $s, t$ are two real numbers which are algebraically independent over $\Q$.
	\end{assumption}
	
	There are many such $s$ and $t$. 
	As $s$ and $t$ are algebraically independent 
	over $\Q$, 
	the elliptic curves $E$ and $F$ are not isogenous.

	As $s,t$ are \textit{real numbers} 
	and the constructions in Section~\ref{sct4} 
	are compatible with field extensions,
	the curves $E$ and $F$ each have a natural real structure,
	denoted by $\imath_E$ and $\imath_F$, 
	and thus each variety $V$ in Section~\ref{sct4} has an induced privileged real structure, denoted by $\imath_V$.
	
	\subsection{Surface birational to an Enriques surface with infinitely many real forms}\label{sct3sub3}

	Let $A \in D_{23} = \pi(C_{23})$. 
	We will work under the following assumption.

	\begin{assumption}\label{ass32} 
	\hfill

		\begin{enumerate}
			\item $A$ is a real point of $D_{23}$, in the sense that $A \in {D_{23}}^{\imath_{Z}}$; 
			\item $A \not\in D_{23} \cap C$ for any irreducible curve $C \subset Z$ with $C \not= D_{23}$ and $C^2 < 0$;
			\item $A \not\in {D_{23}}^{g}$ for any $g \in {\rm Dec}(Z, D_{23}) \setminus {\rm Ine}(Z, D_{23})$.
		\end{enumerate}
	\end{assumption}
	
	The next lemma, which is similar to \cite[Lemma 2.4]{DOY23}, is also crucial in this paper.
	
	\begin{lemma}\label{lem33} 
	There are uncountably many points $A \in D_{23}$ satisfying Assumption \ref{ass32}.
	\end{lemma}
	
	\begin{proof} 
	Note that there are at most countably many irreducible curves $C \not= D_{23}$ on $Z$ with $C^2 < 0$, 
	and thus the points $B \in D_{23}$ which are in the union of $D_{23} \cap C$ (for all such curves $C$) are countable. 
	Note also that ${\rm Aut} (Z)$ is discrete, hence countable, 
	and $D_{23}^{g}$ is at most two points for each $g \in {\rm Dec}(Z, D_{23}) \setminus {\rm Ine}(Z, D_{23})$ because $D_{23} \simeq {\mathbb P}^1$. 
	Therefore the points $B \in D_{23}$ which are in the union of all ${D_{23}}^{g}$, for $g \in {\rm Dec}(Z, D_{23}) \setminus {\rm Ine}(Z, D_{23})$, are also countable. 
	On the other hand, ${D_{23}}^{\imath_{Z}}$ is the set of real points on 
	a real rational curve, which 
	is uncountable. 
	Hence there are uncountably many points $A \in D_{23}$ satisfying Assumption~\ref{ass32}. 
	\end{proof}

	Our main theorem is the following,
	which implies Theorem~\ref{thm11}.
	
	\begin{thm}\label{thm32} 
	Let $s$, $t$ be as in Assumption \ref{ass31} and let $A \in D_{23} \subset Z$ be as in Assumption \ref{ass32}. Let $\mu : Y \to Z$ be the blow-up of $Z$ at $A$. 
	Then 
	\begin{enumerate}
	    \item $Y$ has infinitely many mutually non-isomorphic 
		real forms. 
		\item $\Aut(Y)$ is not finitely generated.
	\end{enumerate}
	\end{thm}

	\begin{remark}\label{rem_dense} By construction, $Y$ in Theorem \ref{thm32} is parametrized by the three real parameters 
	\[ (s, t, A)
	\]
	which move in a dense subset of $\R^3$. 
	\end{remark}

	We will reduce the proof to a problem on the existence of a set of involutions on $X$ with certain properties 
	(Lemma~\ref{lem-critinf}), 
	which we will solve based on results proven in
	Sections~\ref{sec-Kummer} and~\ref{sct4}.

    \ssec{Lifting $\Aut(Y)$ to $\Ine(X,C_{23})$}	
	
	Note that ${\rm Bir} (Z) = {\rm Aut}(Z)$ as $Z$ is a minimal projective smooth surface. Let $E_A$ be the exceptional curve of the blow-up $\mu : Y \to Z$ at the point $A\in Z$.
	Then $|2K_Y| = \{2E_A\}$. Thus under the natural inclusion 
	$${\rm Aut}(Y) \subset {\rm Bir} (Z) = {\rm Aut}(Z),$$
	induced from $\mu$, we have
	$${\rm Aut}(Y) = {\rm Dec}(Y, E_A) = {\rm Ine}(Z, A).$$
	
	If $g \in {\rm Dec} (Z, D_{23})$, then $g$ lifts in two ways to ${\rm Aut}(X)$. Namely, if we write one of them as $\tilde{g}$, then they are $\tilde{g}$ and $\epsilon \circ \tilde{g}$. 
	Note that $\epsilon(C_{32}) = C_{23}$ 
	by Lemma~\ref{lem-eps} (2),
	so $\tilde{g}$ satisfies either $\tilde{g}(C_{23}) = C_{32}$ or $C_{23}$, and hence $\epsilon \circ \tilde{g}(C_{23}) = C_{23}$ or $C_{32}$, respectively. 
	We thus identify $\Dec(Z,D_{23})$ with a subgroup of $\Dec(X,C_{23})$
	through 
	$${\rm Dec} (Z, D_{23}) \hto {\rm Dec} (X, C_{23}) \subset {\rm Aut}(X),$$
	sending $g \in \Dec(Z,D_{23})$ to its unique lifting $\ti{g} \in \Aut(X)$
	satisfying $\ti{g}(C_{23}) = C_{23}$.
	Under such an identification, we have
	$${\rm Ine} (Z, D_{23}) \subset {\rm Ine} (X, C_{23}) \subset {\rm Aut}(X).$$
	
	\begin{lemma}\label{lem-liftIneX} 
	Suppose that $A \in D_{23}$ satisfies Assumption \ref{ass32}.
	We have 
			$${\rm Aut}(Y) = {\rm Ine}(Z, A) = {\rm Ine}(Z, D_{23}) \subset {\rm Ine} (X, C_{23}).$$
	\end{lemma}
	
	\begin{proof} As already remarked, we have 
		$${\rm Aut}(Y) = {\rm Ine}(Z, A), \ \ {\rm Ine}(Z, D_{23}) \subset {\rm Ine} (X, C_{23}).$$
		So, it suffices to show the equality ${\rm Ine}(Z, A) = {\rm Ine}(Z, D_{23})$. 
		
		Since $A \in D_{23}$, we have ${\rm Ine}(Z, D_{23}) \subset {\rm Ine}(Z, A)$. 
		To show the reverse inclusion ${\rm Ine}(Z, A) \subset {\rm Ine}(Z, D_{23})$, 
		let $g \in {\rm Ine}(Z, A)$. 
		Then $A \in D_{23} \cap g(D_{23})$. 
		Since 
		$$g(D_{23})^2 = D_{23}^2 = -2 < 0,$$
		we have $g(D_{23}) = D_{23}$ by Assumption~\ref{ass32} (2). Thus $g \in {\rm Dec} (Z, D_{23})$. As $g \in {\rm Ine}(Z, A)$, we have $A \in {{D}_{23}}^{g}$. 
		Thus $g \in {\rm Ine}(Z, D_{23})$ by 
		Assumption~\ref{ass32} (3). 
	\end{proof}
	
	\ssec{Every automorphism on $Y$ is real}
	
	Recall that we have privileged real structures
	$\imath_{X}$ and $\imath_{Y}$
	on $X$ and $Y$ respectively.
	
	\begin{lemma} \label{lem-real}
	For any $g \in {\rm Aut} (X)$, we have $\imath_{X} \circ g \circ \imath_{X} = g$. In other words, every $g \in {\rm Aut}(X)$ is defined over ${\mathbb R}$, with respect to $\imath_{X}$.
	
	As a consequence,
	the conjugate action of the real structure $\imath_{Y}$ of $Y$
			is trivial on ${\rm Aut}(Y)$.
	\end{lemma}
	
	\begin{proof}
	
    By Assumption~\ref{ass31},
    the elliptic curves $E$ and $F$ are not isogenous, so 
    $$\rho(X) = 18 = \rho(\wt{E \times F})$$
    where $\wt{E \times F}$ 
    is the blowup at the $16$ 
    two-torsion points of $E \times F$ 
    (see e.g.~\cite[Page 389, (1.2)]{Hu16}).
    It follows that
    $\Pic (X) \otimes_{\mathbb Z} {\mathbb Q}$ 
	is generated by
    the 24 smooth rational curves 
	in Figure~\ref{fig1}. 
    As $s,t \in \R$,
    these 24 curves 
	are invariant under $\imath_{X}$.
	Thus for $g \in {\rm Aut}(X)$, the actions of
	$g$ and	$\imath_{X} \circ g \circ \imath_{X}$
	on ${\rm Pic} (X)$ coincide.

	As $H^0(X,\gO_X^2) = \C \cdot \go_X$ and
	$\imath_{X}^*\omega_X = \omega_X$ by construction,
	the actions of $g$ and	
	$\imath_{X} \circ g \circ \imath_{X} \in \Aut(X)$
	also agree on the transcendental part of $H^2(X,\Z)$ 
	(see e.g., \cite[Remark 15.1.2]{Hu16}).
    Since $\Aut(X)$ acts fatihfully on $H^2(X,\Z)$
    (see e.g.~\cite[Proposition 15.2.1]{Hu16}),
    we have 
		$$\imath_{X} \circ g \circ \imath_{X} = g.$$ 
	
	The last statement follows from the facts that $\imath_{X}$ acts on ${\rm Aut}(X)$ as the identity and that the inclusion 
	$\Aut(Y) \subset \Ine(X,C_{23})$
	in Lemma~\ref{lem-liftIneX} 
	is equivariant with respect to the actions defined by
	$\imath_{X}$ and $\imath_{Y}$ by construction. 
	\end{proof}
	
	\ssec{Infinitely many real forms}
	
	\begin{lemma}\label{lem-critinf}
	Assume that there is a set ${\mathcal S} \subset {\rm Ine} (X, C_{23})$ consisting of some involutions on $X$ satisfying the following properties. 
	\begin{enumerate}
	    \item The set of conjugacy classes of ${\mathcal S}$ in ${\rm Ine} (X, C_{23})$ 
			is an infinite set.
		\item Each element of $\cS$ descends to an automorphism on $Y$.
	\end{enumerate}
	Then $Y$ has infinitely many mutually non-isomorphic real forms.
	\end{lemma}
	
	\begin{proof}
	    Let ${\mathcal S}_Y \subset {\rm Aut}(Y)$ 
		be the set of all involutions (including the trivial one) 
		in $\Aut(Y)$. 
		By Proposition \ref{prop21} (4) and Lemma~\ref{lem-real}, 
		we have a one-to-one correspondence between 
		the real forms on $Y$ up to isomorphisms, 
		and the conjugacy classes of ${\mathcal S}_Y$ with respect to ${\rm Aut}(Y)$. 
		Under the inclusion ${\rm Aut}(Y) \subset {\rm Ine} (X, C_{23})$ in Lemma~\ref{lem-liftIneX},
		we have
		${\mathcal S} \subset {\mathcal S}_Y$ by Assumption (2).
		So the cardinality of the conjugacy classes of ${\mathcal S}_Y$ with respect to ${\rm Aut}(Y)$ is larger than or equal to the cardinality of the conjugacy classes of ${\mathcal S}$ with respect to ${\rm Ine} (X, C_{23})$, which is infinite by Assumption (1). 
		Hence $Y$ has infinitely many mutually non-isomorphic real forms.
	\end{proof}

	\begin{proof}[Proof of Theorem~\ref{thm32} (1)]
	
	Consider the set
		$$\cS = \{\psi_n\,|\, n \in {\mathbb Z}\} \subset {\rm Aut}(X).$$
	constructed in Subsection~\ref{subsec:second_jac}.
	By Proposition~\ref{prop43} (2), 
	$\cS$ satisfies Assumption (1) in Lemma~\ref{lem-critinf}.	
	By Lemma~\ref{lem-descent-psin},
	each $\psi_n$ descends to an automorphism on $Z$,
	fixing $D_{23}$ pointwise. Then by Lemma \ref{lem-liftIneX}, each $\psi_n$ descends to an automorphism on $Y$. 
	Thus $\cS$ satisfies Assumption (2) in Lemma~\ref{lem-critinf}.
	We then conclude by Lemma~\ref{lem-critinf}.
	\end{proof}

\ssec{Non-finite generation}

Finally we prove the non-finite generation of ${\rm Aut}(Y)$. 

\begin{proof}[Proof of Theorem~\ref{thm32} (2)]
Let $\Aut^s(X)$ be the subgroup of $\Aut(X)$ 
preserving a holomorphic symplectic form $\go_X$ of $X$
and let
$$\Ine^s(X,C_{23}) \cnec \Ine(X,C_{23}) \cap \Aut^s(X).$$
Since $\Aut^s(X)$ has finite index in $\Aut(X)$ by \cite[Corollary 15.1.10]{Hu16}, 
identifying $\Aut(Y)$ as a subgroup of $\Aut(X)$ through 
Lemma~\ref{lem-liftIneX}, the subgroup
$$\Aut^s(Y) \cnec \Aut(Y) \cap \Ine^s(X,C_{23}) 
= \Aut(Y) \cap \Aut^s(X)$$
also has finite index in $\Aut(Y)$ (see \cite[(3.13)(i)]{Su82}).

Note that for every $g \in \Aut(X)$, we have 
$$g\left(\bigcup^3_{i=0} (E_i\cup F_i)\right) = \bigcup^3_{i=0} (E_i\cup F_i)$$
by \cite[Lemma 1.4]{Og89}. For every $g \in \Ine(X,C_{23})$, since 
$$P_{23} = E_2 \cap C_{23} \in g(E_2) \cap E_2$$
and 
$E_2$ is the unique irreducible component of $\cup^3_{i=0} (E_i\cup F_i)$ containing the point $P_{23}$, necessarily $g(E_2) = E_2$. This gives rise to a homomorphism
$$\rho: \Ine^s(X,C_{23}) \to \Ine(E_2,P_{23}).$$
As $g \in \Ine(X,C_{23})$ preserves the tangent direction $T_{C_{23},P_{23}}$ and acts trivially on it, we see that, assuming further that $g\in\Ine^s(X,C_{23})$, we get a trivial action on $T_{X,P_{23}}$,
and thus on $T_{E_{2},P_{23}}$.
Hence, under the affine coordinate $x$
on $E_2$ defined by
$$x(E_2 \cap C_{20}) = 0, \ \ x(E_2 \cap C_{21}) = 1,  \ \ x(E_2 \cap C_{23}) = \infty$$
(see (\ref{eqn-xE})), we have
$$\rho(g): x \mapsto x + c.$$
We can therefore identify $\rho(\Aut^s(Y))$ 
with a subgroup $G$ of $(\C,+)$.

Since $\psi:X \to X$ is antisymplectic, each $\psi_n = f^{-4n}\circ \psi \circ f^{4n}$ is antisymplectic as well.
Note that $\psi_n \in \Aut(Y)$ 
as we saw in the proof of Theorem~\ref{thm32}, under the inclusion
$\Aut(Y) \subset \Aut(X)$ mentioned previously). Hence,
we have
$\psi_m\psi_n \in \Aut^s(Y)$
for every $m,n\in \Z$. 
As $\psi_m(x) = \frac{1}{r^m} - x$ by~\eqref{eqn-psicoord}, we have
$$\gO \cnec \Set{\frac{1}{r^n} - \frac{1}{r^m} | m,n\in \Z} \subset G.$$
Viewing the abelian group $\gO$ as a $\Z$-module, the transcendence of $r$ yields that $\gO$ contains infinitely many elements that are $\Z$-linearly independent. Hence, by the structure theorem for finitely generated abelian groups, $\gO$ is not finitely generated.
the subgroup of $G$ generated by $\gO$ is not finitely generated.
By the structure theorem for finitely generated abelian groups, every subgroup of a finitely generated abelian group is finitely generated. So, since $G$ is abelian, $G$ itself cannot be finitely generated. As $G$ is a quotient of $\Aut^s(Y)$, we see that $\Aut^s(Y)$ is not finitely generated.
Finally by Schreier's lemma, since $\Aut^s(Y)$ has finite index in $\Aut(Y)$,
the group $\Aut(Y)$ is not finitely generated either.
\end{proof}

	
	

	
	
\end{document}